\documentclass[11pt]{amsart}

\usepackage{amssymb}

\newcommand{\dbar}{\ensuremath{\overline\partial}}
\newcommand{\dbarstar}{\ensuremath{\overline\partial^*}}
\newcommand{\C}{\ensuremath{\mathbb{C}}}
\newcommand{\R}{\ensuremath{\mathbb{R}}}
\newcommand{\D}{\ensuremath{\mathbb{D}}}

\makeatletter
\newcommand{\sumprime}{\if@display\sideset{}{'}\sum%
            \else\sum'\fi}
\makeatother

\begin{document}

\numberwithin{equation}{section}

\newtheorem{theorem}{Theorem}[section]
\newtheorem{proposition}[theorem]{Proposition}
\newtheorem{conjecture}[theorem]{Conjecture}
\def\theconjecture{\unskip}
\newtheorem{corollary}[theorem]{Corollary}
\newtheorem{lemma}[theorem]{Lemma}
\newtheorem{observation}[theorem]{Observation}
\theoremstyle{definition}
\newtheorem{definition}{Definition}
\numberwithin{definition}{section}
\newtheorem{remark}{Remark}
\def\theremark{\unskip}
\newtheorem{question}{Question}
\def\thequestion{\unskip}
\newtheorem{example}{Example}
\def\theexample{\unskip}
\newtheorem{problem}{Problem}

\def\spc{strictly pseudo-convex}
\def\psh{plurisubharmonic}
\def\spsh{strictly plurisubharmonic}
\def\nbd{neighborhood\ }
\def\nbds{neighborhoods\ }
\def\iff{if and only if }
\def\comtanX{C\otimes T(X)}
\def\comtanO{C\otimes T(\Omega)}
\def\dpartial{\overline{\partial}}
\def\Apc{Andreotti pseudoconcave\ }

\def\intprod{\mathbin{\lr54}}
\def\reals{{\mathbb R}}
\def\integers{{\mathbb Z}}
\def\Z{{\mathbb Z}}
\def\D{{\mathbb D}}
\def\naturals{{\mathbb N}}
\def\N{{\mathbb N}}
\def\P{{\mathbb P}}
\def\complex{{\mathbb C}\/}
\def\distance{\operatorname{distance}\,}
\def\crit{\operatorname{crit}\,}
\def\spec{\operatorname{Spec}\,}
\def\sgn{\operatorname{sgn}\,}
\def\hom{\operatorname{Hom}\,}
\def\trace{\operatorname{Tr}\,}
\def\support{\operatorname{Supp}\,}
\def\supp{\operatorname{Supp}\,}
\def\ZZ{ {\mathbb Z} }
\def\e{\varepsilon}
\def\p{\partial}
\def\rp{{ ^{-1} }}
\def\Re{\operatorname{Re\,} }
\def\Im{\operatorname{Im\,} }
\def\dbarb{\bar\partial_b}
\def\eps{\varepsilon}

\def\vvv{\ensuremath{\mid\!\mid\!\mid}}
\def\intprod{\mathbin{\lr54}}
\def\reals{{\mathbb R}}
\def\integers{{\mathbb Z}}
\def\N{{\mathbb N}}
\def\complex{{\mathbb C}\/}
\def\distance{\operatorname{distance}\,}
\def\spec{\operatorname{spec}\,}
\def\interior{\operatorname{int}\,}
\def\trace{\operatorname{tr}\,}
\def\cl{\operatorname{cl}\,}
\def\essspec{\operatorname{esspec}\,}
\def\range{\operatorname{\mathcal R}\,}
\def\kernel{\operatorname{\mathcal N}\,}
\def\dom{\operatorname{\mathcal D}\,}
\def\linearspan{\operatorname{span}\,}
\def\lip{\operatorname{Lip}\,}
\def\sgn{\operatorname{sgn}\,}
\def\Z{ {\mathbb Z} }
\def\e{\varepsilon}
\def\p{\partial}
\def\rp{{ ^{-1} }}
\def\Re{\operatorname{Re\,} }
\def\Im{\operatorname{Im\,} }
\def\dbarb{\bar\partial_b}
\def\OL{\overline{L}}

\def\Hs{{\mathcal H}}
\def\E{{\mathcal E}}
\def\scriptu{{\mathcal U}}
\def\scriptr{{\mathcal R}}
\def\scripta{{\mathcal A}}
\def\scriptg{{\mathcal G}}
\def\scripti{{\mathcal I}}
\def\scriptl{{\mathcal L}}
\def\scripth{{\mathcal H}}
\def\scriptn{{\mathcal N}}
\def\scripte{{\mathcal E}}
\def\scripts{{\mathcal S}}
\def\scriptt{{\mathcal T}}
\def\scriptb{{\mathcal B}}
\def\scriptf{{\mathcal F}}
\def\scripto{{\mathfrak o}}
\def\scriptv{{\mathcal V}}
\def\frakg{{\mathfrak g}}
\def\frakG{{\mathfrak G}}

\def\ov{\overline}
\date {Printed \today.}

\author{Siqi Fu and Howard Jacobowitz}

\begin{abstract}
We study spectral behavior of the complex Laplacian
on forms with values in the $k^{\text{th}}$ tensor power of a holomorphic
line bundle over a smoothly bounded domain with degenerated boundary
in a complex manifold. In particular, we prove that in the two dimensional case,
a pseudoconvex domain is of finite type if and only if for any positive constant
$C$, the number of eigenvalues of the $\overline\partial$-Neumann Laplacian 
less than or equal to $Ck$ grows polynomially as $k$ tends to infinity.
\end{abstract}

\thanks
{Research supported in part by NSF grants.}
\address{Department of Mathematical Sciences,
Rutgers University, Camden, NJ 08102}
\email{sfu@camden.rutgers.edu, jacobowi@camden.rutgers.edu}
\title[]  
{The $\dbar$-cohomology groups, holomorphic Morse inequalities, and finite type conditions} \maketitle

\noindent{\it Dedicated to Professor J. J. Kohn on the occasion of his $75^{\text{\rm th}}$ birthday.}

\tableofcontents

\section{Introduction}\label{sec:intro}

A classical theorem of Siegel \cite{Siegel} says that the dimension of
global holomorphic sections of the $k^{\text{th}}$ tensor power $E^k$
of a holomorphic line bundle $E$ over a compact complex manifold $X$ of dimension $n$
grows at a rate of at most $k^n$ as $k$ tends to infinity. This theorem
has important implications in complex algebraic geometry. For example, Siegel proved that, as a consequence, the algebraic degree of $X$ ({\it i.e.}, the transcendence degree of the field of meromorphic functions on $X$) is less than or equal to $n$. (We refer the reader to \cite{A2} for an exposition of relevant results.)

The classical Morse inequalities on compact Riemannian manifolds,
relating the Betti numbers to the Morse indices, showcase
interplays among analysis, geometry, and topology ({\it e.g.},
\cite{Milnor63}). In an influential paper \cite{Witten82}, Witten
provided an analytic approach to the Morse inequalities.
Instead of studying the deRham complex directly,
Witten used the twisted deRham complex $d_{tf}=e^{-tf}d e^{tf}$, where
$d$ is the exterior differential operator and $f$ the Morse function.
The Morse inequalities then follows from spectral analysis of the twisted
Laplace-Beltrami operator by letting $t\to\infty$. (See, for
example, \cite{HelfferSjostrand85, Bismut86, Zhang01,
HelfferNier05} and references therein for detailed expositions of
Witten's approach.)

Asymptotic Morse inequalities for compact
complex manifolds were established by Demailly (\cite{Demailly85}; see
also \cite{Demailly89}). Demailly's Morse inequalities were inspired in part
by Siu's solution \cite{Siu85} to the Grauert-Riemenschneider
conjecture \cite{GR70} which states that a compact complex manifold with
a semi-positive holomorphic line bundle that is positive on a dense subset
is necessarily Moishezon ({\it i.e.}, its algebraic degree is the
same as the dimension of the manifold). It is
noteworthy that whereas the underpinning of Witten's approach is a
semi-classical analysis of Schr\"{o}dinger operators without
magnetic fields, Demailly's holomorphic Morse inequality is
connected to Schr\"{o}dinger operators with strong magnetic
fields. (Interestingly, a related phenomenon also occurs in
compactness in the $\dbar$-Neumann problem for Hartogs domains in
$\C^2$ (see \cite{FuStraube02,ChristFu05}): Whereas Catlin's
property ($P$) can be phrased in terms of semi-classical limits of
non-magnetic Schr\"{o}dinger operators, compactness of the
$\dbar$-Neumann operator reduces to Schr\"{o}dinger operators with
degenerated magnetic fields.) More recently, Berman established a local version of
holomorphic Morse inequalities on compact complex manifolds \cite{Berman04} and generalized
Demailly's holomorphic Morse inequalities to complex manifolds with non-degenerated boundaries \cite{Berman05}.

Here we study spectral behavior of the complex Laplacian for a relatively compact
domain in a complex manifold whose boundary has a  degenerated Levi form. In particular, we
are interested in Siegel type estimates for such a domain.
Let $\Omega\subset\subset X$ be a domain with smooth boundary in a complex manifold of dimension $n$.  Let $E$ be a holomorphic line bundle over $\Omega$ that extends smoothly to
$b\Omega$. Let $h^{q}(\Omega, E)$ be the dimension
of the Dolbeault cohomology group on $\Omega$ for $(0, q)$-forms with values
in $E$. Let $\widetilde{h}^{q}(\Omega, E)$ be the dimension of the corresponding $L^2$-cohomology group for the $\dbar$-operator (see Section~\ref{sec:prelim} for
the precise definitions). It was proved by H\"{o}rmander that when $b\Omega$
satisfies conditions $a_q$ and $a_{q+1}$\footnote{Recall that the boundary
$b\Omega$ satisfies condition $a_q$ if the Levi form of its
defining function has either at least $q+1$ negative eigenvalues
or at least $n-q$ positive eigenvalues at every boundary point}, then these two cohomology groups
are isomorphic. Furthermore, there exists a defining function $r$ of $\Omega$ and a constant $c_0$, independent of $E$, such that these cohomology groups are isomorphic to their counterparts
on $\Omega_c=\{z\in\Omega \mid r(z)<-c\}$ for all $c\in (0, c_0)$.
Our first result is an observation that combining H\"{o}rmander's theorems \cite{Hormander65} with a theorem of Diederich and Forn{\ae}ss \cite{DiederichFornaess77} yields the following:

\begin{theorem}\label{th:maintheorem1}
 Let $\Omega$ be pseudoconvex. Assume that there exists a neighborhood
 $U$ of $b\Omega$ and a bounded continuous function whose complex hessian is bounded from below by a positive constant on $U\cap\Omega$.  Then $h^q(\Omega, E)=\widetilde{h}^q(\Omega, E)$ for
all $1\le q\le n$. Furthermore, there exists a defining
function of $\Omega$ and a constant $c_0>0$, independent of $E$,
such that $b\Omega_c$ is strictly pseudoconvex and $h^q(\Omega,
E)=\widetilde h^q(\Omega_c, E)$ for all $c\in (0, c_0)$.
\end{theorem}

It is well known that a smooth pseudoconvex domain of finite type in a complex surface satisfies the assumption in the above theorem (\cite{Catlin89, FornaessSibony89}; see
Section~\ref{sec:equiv}). Together with Berman's result, one then obtains a holomorphic Morse inequality for such pseudoconvex domains.  In particular, $h^q(\Omega, E^k)\le C k^n$, $1\le q\le n$, for some constant $C>0$.  For pseudoconcave domains, we have

\begin{theorem}\label{th:maintheorem2} Assume that $\Omega$ is
pseudoconcave and $b\Omega$ does not contain the germ of any  complex hypersurface.
Then $h^0(\Omega, E^k)\le Ck^n$ for some constant $C>0$.
\end{theorem}

Notice that the above theorems show that $h^q(\Omega, E^k)$ is insensitive
to the order of degeneracy of the Levi form of the boundary. This is related to the fact that the dimensions of cohomology groups (equivalently, the multiplicity of the zero eigenvalues of the $\dbar$-Neumann Laplacian) alone, even though can characterize pseudoconvexity (see \cite{Fu05a} for a discussion on related results),  are not sufficient to detect other geometric features, such as the finite type conditions, of the boundary. For this, we need to consider higher eigenvalues. Let $N_k(\lambda)$ be the number of eigenvalues that are less than or equal to $\lambda$ of the $\dbar$-Neumann Laplacian on $\Omega$ for $(0, 1)$-forms with values in $E^k$. The following is the main theorem of the paper:

\begin{theorem}\label{main-theorem}
Let $\Omega\subset\subset X$ be pseudoconvex domain with smooth boundary in a complex surface $X$. Let $E$ be a holomorphic line bundle over $\Omega$ that extends smoothly to $b\Omega$.
Then  for any $C>0$, $N_k(Ck)$ has at most polynomial growth as $k\to\infty$ if and only if $b\Omega$ is of finite type.
\end{theorem}

The proof of the above theorem is a modification of the arguments
in \cite{Fu05}; we need only to establish here that effects of
the curvatures of the metrics on the complex surface $X$ and the
line bundle $E^k$ are negligible.

Our paper is organized as follows.
In Section~\ref{sec:prelim}, we review definitions and notations,
and provide necessary backgrounds. we prove Theorem~\ref{th:maintheorem1}  in Section~\ref{sec:equiv} and Theorem~\ref{th:maintheorem2} in Section~\ref{sec:siegel}.
The rest of the paper is devoted to the proof of Theorem~\ref{main-theorem}. For the reader's convenience, we have made an effort to have the paper self-contained. This results in including previously known arguments in the paper.

\section{Preliminaries}\label{sec:prelim}

\subsection{The $\dbar$-Neumann Laplacian}

We first review the well-known setup for the $\dbar$-Neumann
Laplacian on complex manifolds. (We refer the reader to
\cite{Hormander65, FollandKohn72, ChenShaw99} for extensive
treatises of the $\dbar$-Neumann problem and to \cite{Demailly}
for $L^2$-theory of the $\dbar$-operator on complex manifolds.)
Let $X$ be a complex manifold of dimension $n$. Let $E$ be a
holomorphic vector bundle of rank $r$ over $X$. Let
\[
C^\infty_{0, q}(X, E)=C^\infty(X, \Lambda^{0, q}T^*X\otimes E),
\quad C^\infty_{0, *}(X, E)=\oplus_{q=0}^n C^\infty_{0, q}(X, E).
\]
Let $\theta\colon E|_U\to U\times \C^r$ be a local holomorphic
trivialization of $E$ over $U$. Let $\eps_j$, $1\le j\le r$, be
the standard basis for $\C^r$ and let $e_j=\theta^{-1}(x,
\eps_j)$, $1\le j\le r$,  be the corresponding local holomorphic
frame of $E|_U$. For any $s\in C^\infty_{0, q}(U, E)$, we make the
identification
\[
s=\sum_{j=1}^r s_j\otimes e_j\simeq_\theta (s_1, \ldots, s_r),
\]
where $s_j\in C^\infty_{0, q}(U, \C)$, $1\le j\le r$. Then the
{\it canonical $(0, 1)$-connection} is by definition given by
\[
\dbar_q s\simeq_\theta(\dbar_q s_1, \ldots, \dbar_q s_r)\in
C^\infty_{0, q+1}(M, E),
\]
where $\dbar_q$ is the projection of the exterior differential
operator onto $C^\infty(X, \Lambda^{0, q}T^*X)$.

Now assume that $X$ is equipped with a hermitian metric $h$, given
in local holomorphic coordinates $(z_1, \ldots, z_n)$ by
\[
h=\sum_{j,k=1}^n h_{jk} dz_j\otimes d\bar z_k,
\]
where $(h_{jk})$ is a positive hermitian matrix. Let
$\Omega$ be a domain in $X$.
Let $E$ be a holomorphic vector bundle over $\Omega$ that extends
smoothly to $b\Omega$. Assume that $E$ is equipped with a smoothly
varying hermitian fiber metric $g$, given in a local holomorphic
frame $\{e_1, \ldots, e_r\}$ by $g_{jk}=\langle e_j, e_k\rangle$.
For $u, v\in \C^\infty_{0, *}(X, E)$, let $\langle u, v\rangle$ be the
point-wise inner product of $u$ and $v$, and let
\[
\langle\langle u, v\rangle\rangle_\Omega=\int_\Omega \langle u,
v\rangle dV
\]
be the inner product of $u$ and $v$ over $\Omega$.  Let $L^2_{0,
q}(\Omega, E)$ be the completion of the restriction of
$C^\infty_{0, q}(X, E)$ to $\Omega$ with respect to the inner
product $\langle\langle \cdot, \cdot\rangle\rangle_\Omega$.  We
also use $\dbar_q$ to denote the closure of $\dbar_q$ on
$L^2_{0, q}(\Omega, E)$. Thus $\dbar_q\colon L^2_{0, q}(\Omega,
E)\to L^2_{0, q+1}(\Omega, E)$ is a densely defined, closed
operator on Hilbert spaces. Let $\dbarstar_q$ be its Hilbert space
adjoint.

Let
\[
Q^E_{\Omega, q}(u, v)=\langle\langle\dbar_q u, \dbar_q
v\rangle\rangle_\Omega+\langle\langle\dbarstar_{q-1} u,
\dbarstar_{q-1} v\rangle\rangle_\Omega
\]
be the sesquilinear form on $L^2_{0, q}(\Omega, E)$ with domain
$\dom(Q^E_{\Omega, q})=\dom(\dbar_q)\cap \dom(\dbarstar_{q-1})$.
Then $Q^E_{\Omega, q}$ is densely defined and closed. It then
follows from general operator theory (see \cite{Davies95}) that $Q^E_{\Omega, q}$
uniquely determines a densely defined self-adjoint operator
$\square^E_{\Omega, q}\colon L^2_{0, q}(\Omega, E)\to L^2_{0,
q}(\Omega, E)$ such that
\[
Q^E_{\Omega, q}(u, v)=\langle\langle u, \square^E_{\Omega, q}
v\rangle\rangle, \qquad \text{for}\ u\in \dom(Q^E_{\Omega, q}),
v\in \dom(\square^E_{\Omega, q}),
\]
and $\dom((\square^E_{\Omega, q})^{1/2})=\dom(Q^E_{\Omega, q})$.
This operator $\square^E_{\Omega, q}$ is called {\it the
$\dbar$-Neumann Laplacian} on $L^2_{0, q}(\Omega, E)$.  It is an
elliptic operator with non-coercive boundary conditions. It
follows from the work of Kohn \cite{Kohn63, Kohn64, Kohn72} and
Catlin \cite{Catlin83, Catlin87} that it is subelliptic when
$\Omega$ is a relatively compact and smoothly bounded pseudoconvex domain
of finite type in the sense of D'Angelo \cite{Dangelo82, Dangelo93}:
There exists an $\eps\in (0, 1/2]$
such that
\[
\|u\|^2_{\eps}\le C(Q^E_{\Omega, q}(u, u)+\|u\|^2)
\]
for all $u\in \dom(Q^E_{\Omega, q})$, where $\|\cdot\|_\eps$
denotes the $L^2$-Sobolev norm of order $\eps$ on $\Omega$.

\subsection{The Dolbeault and $L^2$-cohomology groups} The
Dolbeault and the $L^2$-cohomology groups on $\Omega$ with values
in $E$ are given respectively by
\[
H_{0, q}(\Omega, E)=\frac{\{f\in C^\infty_{0, q}(\Omega, E)\mid
\dbar_q f=0\}}{\{\dbar_{q-1} g \mid g\in C^\infty_{0, q-1}(\Omega,
E)\}}
\]
and
\[\widetilde H_{0, q}(\Omega, E)=\frac{\{f\in L^2_{0,
q}(\Omega, E)\mid \dbar_q f=0\}}{\{\dbar_{q-1} g \mid g\in
\dom(\dbar_{q-1})\}}.
\]
It follows from general operator theory that $\widetilde H_{0,
q}(\Omega, E)$ is isomorphic to $\kernel(\square^E_{\Omega, q})$,
the null space of $\square^E_{\Omega, q}$, when $\dbar_{q-1}$ has
closed range. Furthermore, $\widetilde H_{0, q}(\Omega, E)$ is
finite dimensional when $\square^E_{\Omega, q}$ has compact
resolvent, in particular, when it is subelliptic.  It was shown by
H\"{o}rmander \cite{Hormander65} that when $b\Omega$ satisfies
conditions $a_q$ and $a_{q+1}$,
the $L^2$-cohomology group $\widetilde H_{0, q}(\Omega, E)$ is
isomorphic to the Dolbeault cohomology group $H_{0,
q}(\Omega, E)$.

\subsection{The spectral kernel} Assume that $\square^E_{\Omega, q}$
has compact resolvent. Let $\{\lambda^q_j; j=1, 2, \ldots\}$ be
its eigenvalues, arranged in increasing order and repeated
according to multiplicity. Let $\varphi^q_j$ be the corresponding
normalized eigenforms.  The {\it spectral resolution}
$\mathbf{E}^E_{\Omega, q}(\lambda)\colon L^2_{0, q}(\Omega, E)\to
L^2_{0, q}(\Omega, E)$ of $\square^q_\Omega$ is given by
\[
\mathbf{E}^E_{\Omega, q}(\lambda)u=\sum_{\lambda_j\le \lambda}
\langle\langle u, \varphi^q_j\rangle\rangle \varphi^q_j.
\]
Let $e^E_{\Omega, q}(\lambda; z, w)$ be the spectral kernel, i.e.,
the Schwarz kernel of $\mathbf{E}^E_{\Omega, q}(\lambda)$. Then
\[
\trace e^E_{\Omega, q}(\lambda; z, z)=\sum_{\lambda_j\le \lambda}
|\varphi_j^q(z)|^2.
\]
Let
\[
S^E_{\Omega, q}(\lambda; z)=\sup\{|\varphi(z)|^2 \mid \varphi\in
\mathbf{E}^E_{\Omega, q}(\lambda)(L^2_{0, q}(\Omega, E)), \|\varphi\|=1\}.
\]
It is easy to see that
\begin{equation}\label{eq:e-s-compare}
S^E_{\Omega, q}(\lambda; z)\le \trace e^E_{\Omega, q}(\lambda; z,
z)\le \frac{n!}{q!(n-q)!} S^E_{\Omega, q}(\lambda; z).
\end{equation}
(see, e.g., Lemma~2.1 in \cite{Berman04}.)

\section{Isomorphism between the Dolbeault and $L^2$
cohomology groups}\label{sec:equiv}

The following proposition is a simple variation of a result due to Diederich and Forn{\ae}ss \cite{DiederichFornaess77}.  We provide a proof, following Sibony (\cite{Sibony87, Sibony89}), for completeness.

\begin{proposition}\label{prop:d-f}  Let $\Omega\subset\subset X$
be pseudoconvex with smooth boundary. Assume that there exists a neighborhood $U$ of $b\Omega$ and a bounded continuous function whose complex hessian is bounded below by a positive constant. Then
there exists an $\eta\in (0, 1)$, a smooth defining function $\tilde r$ of $\Omega$, and a constant $C>0$ such that  $\rho=-(-\tilde r)^\eta$ satisfies
\begin{equation}\label{eq:d-f}
L_\rho(z, \xi)=\partial\dbar\rho(\xi, \bar\xi)\ge C|\rho(z)||\xi|^2_h
\end{equation}
for all  $z\in \Omega\cap U$ and $\xi\in T^{1, 0}_z(X)$.
\end{proposition}

\begin{proof} Let $r$ be a defining function of $\Omega$. Let $V\subset\subset U$ be a tubular neighborhood of $b\Omega$ such that the projection from $z\in V$
onto the closest point $\pi(z)\in b\Omega$ is well-defined and smooth. Shrinking $V$ if necessary, we may assume that both $z$ and $\pi(z)$
are contained in the same coordinate patch. By decomposing $\xi\in T^{1, 0}_z(X)$ into complex tangential and normal components, we then obtain that
\begin{equation}\label{eq:levi-normal}
L_r(z, \xi)\ge -C_1(|r(z)||\xi|^2+|\xi||\langle \partial r(z), \xi\rangle|)
\end{equation}
for some constant $C_1>0$.

Write $\rho(z)=\varphi(r(z))e^{f(z)}$ where $\varphi$ is a smooth function on $(-\infty, 0)$ and $f(z)$ on $U$. Then it follows from direct calculations that
\begin{equation}\label{eq:levi-const}
L_\rho(z, \xi)=e^f\left(\varphi' L_r(z, \xi)+\varphi'' |\langle\partial r, \xi\rangle|^2+\varphi L_f(z, \xi)+\varphi |\langle \partial f, \xi\rangle|^2+2\varphi'\Re \langle\partial r, \xi\rangle\overline{\langle \partial f, \xi\rangle}\right).
\end{equation}
Let $\varphi(t)=-(-t)^\eta$. Let $A>0$ be a constant to be determined. Using the inequalities
\[
2|\xi||\langle \partial r, \xi\rangle|\le \frac{A|r|}{\eta}|\xi|^2+\frac{\eta}{A|r|}|\langle\partial r, \xi\rangle|^2
\]
and
\[
2|\langle \partial r, \xi\rangle\overline{\langle\partial f, \xi\rangle}|\le \frac{A|r|}{\eta}|\langle\partial f, \xi\rangle|^2+\frac{\eta}{A|r|}|\langle\partial r, \xi\rangle|^2,
\]
we then obtain from \eqref{eq:levi-normal} and \eqref{eq:levi-const} that
\begin{align}\label{eq:levi-final}
L_\rho(z, \xi)&\ge |\rho|\big( -C_1(\eta+\frac{1}{2}A)|\xi|^2+\frac{\eta}{|r|^2}(1-\eta-\frac{\eta}{A}
-\frac{C_1\eta}{2A})|\langle\partial r, \xi\rangle|^2 \notag\\
&\qquad\quad -L_f(z, \xi)-(1+A)|\langle\partial f, \xi\rangle|^2\big).
\end{align}

By Richberg's theorem, we may assume that there exists a bounded $g\in C^\infty(U\cap\Omega)$ such that $L_g(z, \xi)\ge C|\xi|^2$ for some $C>0$. By rescaling $g$, we may further assume that $-2\le g\le -1$. Let $f=-e^g$. Then on $V$,
\begin{equation}\label{eq:levi-p}
-L_f(z, \xi)-(1+A)|\langle \partial f, \xi\rangle|^2=e^g(L_g(z, \xi)+
(1-(1+A)e^g)|\langle\partial g, \xi\rangle|^2).
\end{equation}
Now choosing $A$ and then $\eta$ sufficiently small, we then
obtain \eqref{eq:d-f} on $V\cap\Omega$. We extend $\rho$ to
a strictly plurisubharmonic function on $U\cap\Omega$ by letting
$\widetilde\rho(z)=\theta (\rho)+\delta \chi (z)g(z)$ where $\theta$
is a smooth convex increasing function such that $\theta(t)=t$ for
$|t|\le\eps$ and is constant when $t<-2\eps$ for sufficiently small $\eps>0$, $\chi (z)\in C^\infty_c(\Omega)$ is identically 1 on a neighborhood of $\Omega\setminus V$, and $\delta$ is sufficiently small.
The desirable defining function is obtained by letting $\widetilde r=-(-\widetilde\rho)^{1/\eta}$.
\end{proof}

Let $c_0>0$ be any sufficiently small constant such that $\{\rho=-c^\eta_0\}\subset U\cap\Omega$ (following the notations of Proposition~\ref{prop:d-f}). Theorem~\ref{th:maintheorem1} is then a consequence of the combination of the above proposition and the results in Chapter III in \cite{Hormander65}\footnote{Although the results in \cite{Hormander65} is stated for only forms with values in the trivial line bundle, it is obvious that they also hold for forms with values in any holomorphic line bundle.}. More specifically, that $h^q(\Omega, E)=\widetilde h^q(\Omega_c, E)$, $1\le q\le n$, for any $c\in (0, c_0)$ follows from Theorem 3.4.9 in \cite{Hormander65}.
The proof of $\widetilde h^q(\Omega, E)=\widetilde h^q(\Omega_c, E)$ also
follows along the line of the proof of Theorem 3.4.9.  We provide details
as follows: Since $\Omega_c$ is strictly pseudoconvex, $\widetilde H_{0, q}(\Omega_c, E)$ is finite dimensional. To prove that the restriction
map $\widetilde H_{0, q}(\Omega, E)\to \widetilde H_{0, q}(\Omega_c, E)$ is onto, one needs only to show that the restriction of the nullspace $\scriptn(\dbar_q, \Omega)$ to $\Omega_c$ is dense in $\scriptn(\dbar_q, \Omega_c)$.  Let $\{c_j\}_{j=1}^\infty$ be an decreasing sequence of positive numbers approaching $0$ with $c_1=c$. Let $f=f_1\in\scriptn(\dbar_q, \Omega_{c_1})$ and let $\eps>0$. By applying Theorem~3.4.7 in \cite{Hormander65} inductively, we obtain $f_j\in \scriptn(\dbar_q, \Omega_{c_{j}})$ such that
\[
\|f_j-f_{j+1}\|_{\Omega_{c_{j}}}\le \frac{\eps}{2^{j}}.
\]
It follows that there exists some  $g\in\scriptn(\dbar_q, \Omega)$ such that for any $k$, $\|f_j-g\|_{\Omega_{c_k}}\to 0$ as $j\to\infty$, and $\|f-g\|_{\Omega_c}\le\eps$. Hence the restriction map is surjective.

To prove the injectivity of the restriction map, it suffices to prove that
for any $f\in L^2_{0, q}(\Omega, E)$ such that $\dbar_q f=0$ on $\Omega$ and $f=\dbar_{q-1} u$ on $\Omega_c$ for some $u\in L^2_{0, q}(\Omega_c, E)$, there exists a form $v\in L^2_{0, q-1}(\Omega, E)$ such that $f=\dbar_{q-1} v$
on $\Omega$. By Theorem~3.4.6 in \cite{Hormander65}, there exists $\tilde v_1\in L^2_{0, q-1}(\Omega_{c_2}, E)$ such that $f=\dbar_{q-1} \tilde v_1$ on $\Omega_{c_2}$. Since $\dbar_{q-1}(u-\tilde v_1)=0$ on $\Omega_{c_1}$, by Theorem~3.4.7 in \cite{Hormander65}, there exists $\hat v_1\in L^2_{0, q-1}(\Omega_{c_2}, E)$ such that $\dbar_{q-1}\hat v_1=0$ on $\Omega_{c_2}$ and $\|u-\tilde v_1-\hat v_1\|^2_{\Omega_{c_1}}<1/2$. Let $v_1=\tilde v_1+\hat v_1$.  Continuing this procedure inductively, we then obtain $v_j\in L^2_{0, q-1}(\Omega_{c_{j+1}}, E)$ such that $f=\dbar_{q-1} v_j$ on $\Omega_{c_{j+1}}$ and
\[
\|v_j-v_{j+1}\|_{\Omega_{c_{j+1}}}\le \frac{1}{2^j}.
\]
It then follows that there exists $v\in L^2_{0, q}(\Omega, E)$ such that for any $k$, $\|v_j-v\|_{\Omega_{c_k}}\to 0$. Hence $f=\dbar_{q-1} v$ on $\Omega$.

\bigskip

\noindent{\bf Remark}. Recall that a domain $\Omega\subset\subset X$ is said to satisfies property ($P$) in the sense of Catlin \cite{Catlin84b} if for any $M>0$, there exists a neighborhood $U$ of $b\Omega$ and a function $f\in C^\infty(\Omega\cap U)$ such that $|f|\le 1$ and
\[
L_f(z, \xi)\ge M |\xi|^2_h
\]
for all $z\in \Omega\cap U$ and $\xi\in T^{1,0}_z(X)$.

It is well-known that any relatively compact smoothly bounded pseudoconvex domain with finite type boundary in a complex surface satisfies property ($P$)(see Catlin \cite{Catlin89} and Fornaess-Sibony \cite{FornaessSibony89}\footnote{Both papers stated their results for domains in $\C^2$. The construction of Fornaess and Sibony can be easily seen to work on complex manifolds as well.  Catlin \cite{Catlin87} also constructed plurisubharmonic function with large complex hessian near the boundary for a smooth bounded pseudoconvex domain of finite type in $\C^n$. As a consequence, such domains satisfy property ($P$). His construction should also work for domains in complex manifolds as well.}).  Hence Theorem~\ref{th:maintheorem1} applies to these domains.

\section{Pseudoconcavity and the asymptotic estimates}\label{sec:siegel}
Let $\Omega\subset\subset X$ be a smoothly bounded domain in a complex manifold
$X$ of dimension $n$.
\begin{definition} A point $q\in b\Omega$ is an Andreotti pseudoconcave point
if there exists a fundamental system of \nbds $\{U\}$ such that $q$ is
an interior point of each of the sets
\[ (\widehat{U\cap \Omega})_U=\{p\in U|\ |f(p)|\leq \sup_{z\in U\cap \Omega}|f(z)|,
\forall f \in \mathcal{O}(U)\}.
\]
A domain $\Omega$ is Andreotti pseudoconcave if each of its boundary
points is.
\end{definition}

Note that this definition does not depend on the choice of the
fundamental system of neighborhoods.

We list some well known properties of Andreotti pseudoconcavity.
\begin{enumerate}
\item There are no relatively compact pseudoconcave domains in $\C^n$.
\item If the Levi form of $b\Omega$ has at each $q\in b\Omega$ at least
one negative eigenvalue, then $\Omega$ is (strictly) pseudoconcave.
\item If $X$ has a relatively compact pseudoconcave subdomain, then $\mathcal{O} (X)=\C$.
\item Each complex submanifold of $\C\P^n$ has a  pseudoconcave
neighborhood.
\end{enumerate}

Recall that a real hypersurface
$M$ in $X$
is minimal (in the sense of Trepreau) at a point $q$ if there does not exist the germ of complex
hypersurface passing through $q$ and contained in $M$.  We say an open
set is minimal at a boundary point if the boundary is minimal at that
point.

\begin{theorem} \label{Apoint} Let $q\in b\Omega$.  If the Levi form has no
positive eigenvalues in a \nbd of $q$ and $\Omega$ is minimal at $q$, then $q$
is an \Apc point.
\end{theorem}
\begin{theorem}\label{basic}
 If $\Omega$ is Andreotti  pseudoconcave and relatively compact and if $E$ is a holomorphic line
 bundle over  $\overline{\Omega}$,
then there is some constant depending only on $\Omega$ and $E$ such that
\[
h^0(\Omega,E^k)\leq Ck^n,
\]
where as before $h^0(\Omega ,E^k)$ is the dimension of the space of global holomorphic sections over $\Omega$ with values in $E^k$.
\end{theorem}
The first theorem will follow directly from  some well-known results. When
$b\Omega$ is real analytic, this theorem is essentially contained
in \cite{BF}.  The second theorem is due to Siegel \cite{Siegel} and
Andreotti \cite{A1} and \cite{A2}.  We sketch the proof from
\cite{A1}, simplified to apply to manifolds rather than spaces.

We state  Trepreau's Theorem \cite{Trepreau}

\begin{theorem}\nonumber  
 If $b\Omega$ is minimal at the point $q$ then there exist a fundamental
system of
\nbds $\{V\}$ of $q$  and an open set $S
  $  lying on one side of $b\Omega$, with $b S\cap b\Omega$ an
open \nbd of $q$ in $b\Omega$, such that
\[
\mathcal{O} (V)\rightarrow \mathcal{O} (V\cap S )
\]
is surjective.
\end{theorem}

To prove Theorem \ref{Apoint}, we first show that the set $S$ from Trepreau's theorem lies in $\Omega$.  To see this, note that as long as $\epsilon$ is small
enough, $b\Omega\cap B(q,\epsilon )$ is Levi pseudoconvex, as part
 of the boundary of $\Omega ^c\cap B(q,\epsilon )$  and so $\Omega ^c\cap
B(q,\epsilon )$ is pseudoconvex and therefore a domain of holomorphy.
Thus for no neighborhood $V$ of $q$ is the map
\[
\mathcal{O} (V)\rightarrow \mathcal{O} (V\cap \Omega ^c )
\]
surjective.  So $S\subset \Omega $ and Trepreau's
theorem  asserts that each holomorphic function on $V\cap
\Omega $, for $V$ in the fundamental family of neighborhoods , extends
holomorphically to $V$.  It then follows that
\[
\sup_V |f|=\sup_{V\cap\Omega} |f|
\]
and so $q$  is a pseudoconcave point. We thus conclude the proof of Theorem~\ref{Apoint}.

Let ${\Omega}\subset\subset X$.  We assume that $\Omega$ is \Apc .  We
may extend the relevant transition functions of the bundle $E$
holomorphically across the boundary of $\Omega$ and so there is no
loss of generality in assuming that
$E$ is a holomorphic line bundle over an open \nbd of
$\overline{\Omega}$.

The following
lemma (see \cite{Siegel}) is fundamental to the arguments.

\begin{lemma}\label{lm:Siegel}
 There exist a finite number of points
  $x_a\in \Omega$, $a=1\ldots N$ and an integer $h$ such that the only
  section of $E$ which vanishes to at least order $h$ at each $x_a$ is
  the zero section.  Further, there exists a constant $C$ depending
  only on a fixed covering of $\Omega$ so that $h$ can be chosen to be any integer greater than
  $C\ln{g}$, where $g$ depends only on bounds for the
  transition functions for $E$.
\end{lemma}

\begin{proof}
Let \[
P_r=\{z\in \C ^n\  |\ \  |z_m|<r,\ \ 1\leq m\leq n\}
\]
be a polydisc.
There exist finite open coverings, $\{\Omega _k,\
k=1,\ldots , K\}$ and $\{ W_a,\  a=1,\ldots , N\}$, of $\overline{\Omega}$ with the following
properties.
\begin{enumerate}
\item Each $\Omega _k $ is diffeomorphic to the unit ball in $\C ^n$ and
  biholomorphic to a domain of holomorphy in $\C ^n$.  Note that any
  holomorphic
  line bundle over $\Omega _k$ is holomorphically trivial.
\item There is some real number $r_0$, $0<r_0<1$ such that for each $a$
  there exists a biholomorphism $\phi _a$ defined on an open \nbd  of
  $\overline{W_a}$ taking  $ W_a\to P_{r_0}$.  It follows that there
  is some real number $r_1$, not depending on $a$ and with
  $r_0<r_1<1$,
 for which $\phi
  _a^{-1}(P_{r_1})$ is defined.  Set
\[
V_a=\phi _a ^{-1} (P_{r_1}).
\]
\item There exists a map
\[ J:\{1,\ldots , N\}\to \{1,\ldots , K\}
\]
such that
\[
V_a
\subset \Omega _{J(a)}\cap\{\widehat{\Omega _{J(a)}\cap \Omega}\}.
\]
\end{enumerate}

Choose a nowhere zero holomorphic section $\sigma _j: \Omega_j\to
E|_{\Omega_j}$.  Define
\[
g_{jk}:\Omega _j\cap \Omega _k\to C
\]
by
\[
 \sigma _j=g_{jk}(x) \sigma_k.
\]
Let
\[
g=\sup_{j,k}\sup_{x\in \Omega _{j}\cap \Omega _{k}}|g_{jk}(x)|.
\]
Note that $1\leq g<\infty $.

Set $\Omega_0=\bigcup V_a$.
For each section $s:\Omega_0\to E$ we introduce the notation
\[ s|_{V_a}=s_{J(a)}\sigma _{J(a)}
\]
and define
\begin{eqnarray*}
||s||_V&=\  \sup_{a,V_{a}}|s_{J(a)}|\\
||s||_W&=\  \sup_{a,W_{a}}|s_{J(a)}|
\end{eqnarray*}
Note that
\[
s_{J(b)}=s_{J(a)}g_{J(a)J(b)} .
\]
Since
\[
V_{J(a)}\subset\widehat {(\Omega _{J(a)}\cap \Omega)}
\]
we have
\begin{equation}\label{xx}
\sup_{V_J(a)}|s_{J(a)}|\leq \sup_{\Omega _J(a)\cap \Omega}|s_{J(a)}|.
\end{equation}
Let $x\in \Omega _{J(a)}\cap \Omega$.  Then, because $\{W_{b}\}$ is a covering
of $\Omega$, there is some $b$ for which $x\in W_{b}$.  So
\[
|s_{J(a)}(x)|=|g_{J(a)J(b)}(x)s_{J(b)}(x)|\leq g\ |s_{J(b)}(x)|.
\]
Combining this with (\ref{xx}), we obtain
\[
||s||_V\leq g\ ||s||_W.
\]
Note how the pseudoconcavity was used to derive this inequality.
Now we need a good bound for $||s||_W $ in terms of $||s||_V$.  This
is the main point in the proof.  Let
\[
x_a=\phi _a^{-1}(0).
\]
We now make use of
the hypothesis that $s$ vanishes to order $h$ at
each $x_a$.  Recall the following version of the Schwarz Lemma \cite{Siegel}.

\begin{lemma} If $F(z) $ is holomorphic on $P_{r_1}$ and vanishes to
  order at least $h$ at the origin, then
\[
\sup_{z\in P_{r_0}}|F(z)|\leq (\frac {r_0}{r_1})^h \sup_{z\in
  P_{r_1}}|F(z)|.
\]
\end{lemma}
Let $q=r_0/r_1$.  Applying the Schwarz Lemma to $W_{a}\subset
V_{a}$ for each $a$, we obtain
\[
||s||_W\leq q^h||s||_V.
\]
Note that $q$ satisfies $0<q<1$ and that $q$ only depends on the
choice of $\{W_a\}$ and $\{\phi _a\}$.

We now have
\[
||s||_V\leq g ||s||_W\leq q^hg ||s||_V
\]
and so $s\equiv 0$ provided we take $h$ to be an integer satisfying
\begin{equation}
h>-\frac {\ln{ g}}{\ln{q}}.\label{h}
\end{equation}
(Recall $g\geq 1$ and $0<q<1$.)
\end{proof}
It is now easy to see how the Fundamental Lemma implies Theorem \ref{basic}.  Let $\Gamma
(\overline{\Omega},E)$ be the complex vector space of holomorphic sections of
$E$ over some \nbd of $\Omega$.  The \nbd is allowed to depend on the
section.  Let $J(x_a)$ be the space of jets up to order $h$ at $x_a$ of
holomorphic sections of $E$.  Consider the map
\[
\Gamma
(\overline{\Omega},E)\to \oplus _{a=1}^N J(x_a).
\]
The Fundamental Lemma tells us that this map is injective.  The
dimension of each $J(x_a)$ is $
\left(\begin{array}{c}
n+h\\
h
\end{array}\right)$.

  Thus
\[
h^0(\Omega,E)\leq N\left(\begin{array}{c}
n+h\\
h
\end{array}\right).
\]

We want to estimate the right hand side when $h$ is large.  We do this
using Sterling's asymptotic formula:
\[
m!\approx \sqrt{2\pi m}\frac {m^m}{e^m}.
\]
So, as $h\to\infty$
\begin{eqnarray*}
\left(\begin{array}{c}
n+h\\
h
\end{array}\right)=
\frac {(n+h)!}{h!n!} & \approx & \frac
      {\sqrt{1+\frac{n}{h}}(n+h)^{n+h}}{e^nh^hn!}\\
& \approx & \frac {(n+h)^n}{n!}\\
& \approx & \frac {h^n}{n!}.
\end{eqnarray*}
Now we replace $E$ by $E^k$.  This means that the transition functions
are replaced by their
$k^{th}$ powers and so $g$ becomes $g^k$.  Then $h$ is replaced by
$c'k$ for some $c'$ depending on $E$ and so
\begin{equation}
 h^0(\Omega,E^k)<C\left(\begin{array}{c}
n+c'k\\
c'k
\end{array}\right)<C'k^n.\label{h00}
\end{equation}

This completes the proof, based on \cite{A1}, of Theorem \ref{basic}.

We will need some minor modifications of Lemma~\ref{lm:Siegel}.
First we replace $E^k$ by a bundle of the form $L^k\otimes F^s$.  The
transition functions for $L^k\otimes F^s$ are of the form
$g_{ij}^kf_{ij}^s$ and so are
bounded by $C^{k+s}$ for some $C$.  So
in the proof of Theorem \ref{basic} the inequality (\ref{h}) is
replaced by
\begin{equation}
h>\frac {-\ln C^{k+s}} {\ln q}\label{hnew}
\end{equation}
and inequality (\ref{h00}) is replaced by
\begin{equation}
 h^0(\Omega,L^k\otimes F^s)<C\left(\begin{array}{c}
n+c'(k+s)\\
c'(k+s)
\end{array}\right)<C'(k+s)^n.
\end{equation}

Next we assume that local holomorphic coordinates $\zeta _1,\ldots , \zeta _n$
are specified in a \nbd of each $x_a$ and that the only sections we
consider are those that in a \nbd of $x_a$ are holomorphic functions
of only $\zeta _1,\ldots , \zeta _m$ for some $m\leq n$.  Denote the
space of such sections by $\Gamma _0(\overline{\Omega },E)$.  We have
restricted the set of sections so of course it still follows that a
section vanishing to at least order $h$ at each $x_a$ must be
identically zero.  At each $x_a$ there are $\left(\begin{array}{c}
m+h\\
h
\end{array}\right)$ polynomials in $m$ variables of degree less than or
equal to $h$.  So
\[
dim\ \Gamma _0(\overline{\Omega },E)\leq N\left(\begin{array}{c}
m+h\\
h
\end{array}\right)
\]
and as before
\[
dim\ \Gamma _0(\overline{\Omega },E^k)\leq C\left(\begin{array}{c}
m+c'k\\
c'k
\end{array}\right)\leq C'k^m.
\]
And finally, we combine these two modifications.
\begin{equation}
dim\ \Gamma _0(\overline{\Omega },L^k\otimes F^s)\label{gamma0}
\leq C'(k+s)^m.
\end{equation}

We conclude with an application from \cite{Siegel}.
Let $X$ be a compact complex manifold with $dim\ X =n$. (Or, more
generally, let $X$ contain a relatively compact Andreotti
pseudoconcave subset, see \cite{A1}.  In particular, the results will
apply to any $X$ containing a set $\Omega$ as in Theorem \ref{Apoint}) .
Recall that
meromorphic functions on a complex manifold $X$ are analytically dependent if
\[
df_1\wedge\ldots\wedge df_{m}= 0
\]
at each point of $X$ at which the functions are all holomorphic.  And
they are algebraically dependent if there is a nontrivial polynomial
$P$ over $C$ with
\[
P(f_1,\ldots ,f_{m})= 0
\]
at all such points.  It is easy to see that algebraic dependence
implies analytic dependence.  Here is the converse.
\begin{theorem}\label{Siegel}
 If the meromorphic functions $f_1,\dots ,f_m$ on $X$ are analytically
dependent, then they are also algebraically dependent.
\end{theorem}
This implies that the field
 of meromorphic functions on $X$ is an algebraic
extension of the field of rational functions in $d$ variables, with
$d\leq n$.  Thus
\[
K(X)\cong Q(t_1,\ldots ,t_d,\theta),\ \theta \mbox{
algebraic in } t_1,\ldots ,t_d.
\]
We first relate memomorphic functions to line bundles.   Given a
meromorphic function $f$, we may find a finite covering
\[
X=\bigcup U_j
\]
and holomorphic functions over $U_j$ such that
\[
 f=\frac {p_j}{q_j} \mbox{ on }U_j,\mbox{  and  } \frac
{p_j}{q_j}=\frac {p_k}{q_k} \mbox{ on } U_j\cap U_k.
\]
Let $L$ be the line bundle with transition functions
\[
g_{jk}=\frac {q_j}{q_k}\in \mathcal{O}^*(U_j\cap U_k).
\]
Then
\[
p=\{p_j\} \mbox{ and } q=\{q_j\}
\]
are global sections of $L$ and $f=p/q$ is a
global quotient.

We are now ready to prove Theorem \ref{Siegel}.  We change notation
and start with analytically independent meromorphic functions
$f_1,\ldots ,f_m$ and a meromorphic function $f$ with
\[
df_1\wedge\ldots\wedge df_{m}= 0
\]
at each point where this makes sense.  We need to find a
polynomial such that $P(f_1,\ldots f_m,f)=0$ at each point where the
functions are all holomorphic.  Let $L_j$ be the bundle associated to
$f_j$ and let $F$ be the bundle associated to $f$.  Then each
 $f_j$ is a global quotient of sections of
\begin{equation*}
L=L_1\otimes L_2\otimes \ldots \otimes
L_m.\label{L}
\end{equation*}
Further, $f_1^{k_1}\cdots f_m^{k_m}f^s$
is a global quotient of $L^k\otimes F^s$ where
$k=k_1+\ldots +k_m$.

We write
\begin{eqnarray*}
f_j &=& \frac {s_j} {s_0} \quad\quad s_j,s_0 \mbox{ global
sections of } L,\\
f &=& \frac {\phi} {\psi } \quad\quad \phi ,\psi \mbox{ global
sections of } F.
\end{eqnarray*}
We fix some positive integers $r$ and $s$.  Let
\begin{eqnarray*}
W_0(r,s) & = & \{\mbox{polynomials of degree at most } r\\
 &   & \mbox{ in each of }
X_1,\ldots , X_m\\
 &   &  \mbox{ and degree at most } s \mbox{ in }
X_{m+1}\}
\end{eqnarray*}
We want to eliminate the denominators
in our global quotients and also to work only with homogeneous polynomials.  So let
\[
W(r,s)=\{Q(\xi , \eta , X_1,\ldots
,X_m,X_{m+1})=  \xi ^{mr}\eta^sP(\frac {X_1}{\xi},\ldots ,\frac
{X_m}{\xi},\frac {X_{m+1}} {\eta}), P\in
W_0(r,s)\}.
\]
Thus $Q\in W(r,s)$ is homogeneous in the sense that
\[
Q(a\xi , b\eta , aX_1,\ldots ,aX_m,bX_{m+1})
=
a^{mr}b^sQ(\xi , \eta , X_1,\ldots
,X_m,X_{m+1}).
\]
We may assume that at the points $x_a$ in the proof of Theorem \ref{basic}
\[
df_1\wedge\ldots\wedge df_{m}\neq 0.
\]
So these functions define a partial set of local coordinates which we
use to define $\Gamma _0(\overline{\Omega}, L^{mr}\otimes F^s)$.

Next define
\[
\Pi :W(r,s)\to \Gamma _0(X, L^{mr}\otimes F^s)
\]
by
\[
\Pi Q = Q(s_0,\psi,s_1,\ldots ,s_m,\phi ).
\]
It suffices to prove that
$\Pi$ is not injective.    The modifications of Theorem \ref{basic}
apply as long as (\ref{hnew}) holds.  Thus (\ref{gamma0}) holds with
$k$ replaced by $mr$:
\[
dim\ \Gamma _0(X, L^{mr}\otimes F^s) \leq C'(mr+s)^m.
\]
It is easy to see that
\[
dim\ W(r,s)=(r+1)^m(s+1).
\]
So if $r$ and $s$ can be chosen such that
\[
(r+1)^m(s+1)>C'(mr+s)^m
\]
then $\Pi$ is not injective.
We write this inequality as
\begin{equation}\label{s+1}
s+1>C\frac {(m+\frac s r)^m }{(1+\frac 1 r )^m}.
\end{equation}
We first choose $s$ so that
\[
s+1>2Cm^m
\]
and then choose $r$ large enough to guarantee (\ref{s+1}).

\section{Hearing the finite type condition in two dimensions}

\subsection{The finite type condition}\label{sec:finite-type}
Hereafter, we will assume that $X$ is a complex surface and
$\Omega$ is a relatively compact domain with smooth boundary in
$X$. The boundary $b\Omega$ is said to be of finite type (in the
sense of D'Angelo \cite{Dangelo82}) if the normalized order of contact of any
analytic variety with $b\Omega$ is finite. The highest order of
contact is the type of the domain.

Assume that $X$ is equipped with a hermitian metric $h$. Let
$r(z)$ be the signed geodesic distance from $z$ to $b\Omega$ such
that $r<0$ on $\Omega$ and $r>0$ outside of $\Omega$. Then $r$ is
smooth on a neighborhood $U$ of $\Omega$ and $|dr|_h=1$ on $U$.
Let $z'\in b\Omega$ and let $L$ be a normalized
$(1, 0)$-vector field in a neighborhood of $z'$ such that $Lr=0$.
For any integers $j, k\ge 1$, let
\[
\scriptl_{jk}\partial\dbar r(z')=\underbrace{L\ldots L}_{
\text{$j-1$ times}}\;
\underbrace{\ov{L}\ldots\ov{L}}_{\text{$k-1$
times}}\partial\dbar r(L, \ov{L})(z'),
\]
Let $m$ be any positive integer. For any $2\le l\le 2m$, let
\begin{equation}\label{al-def}
A_l(z')=\Big(\sum_{\substack{j+k\le l \\ j, k>0}}
|\scriptl_{jk}\partial\dbar r(z')|^2\Big)^{1/2}.
\end{equation}
For any $\tau>0$, let
\begin{equation}\label{delta-def}
\delta(z', \tau)=\sum_{l=2}^{2m} A_l(z')\tau^l.
\end{equation}
It is easy to see that
\begin{equation}\label{delta-comp}
\delta(z', \tau)\lesssim \tau^2\quad\text{and}\quad
c^{2m}\delta(z', \tau)\le \delta(z', c\tau)\le c^2\delta(z',
\tau),
\end{equation}
for any $\tau$ and $c$ such that $0<\tau, c<1$. Furthermore,
$b\Omega$ is of finite type $2m$ if and only if $\delta(z',
\tau)\gtrsim \tau^{2m}$ uniformly for all $z'\in b\Omega$ and
$\delta(z'_0, \tau)\lesssim \tau^{2m}$ for some $z'_0\in b\Omega$.
(Here and throughout the paper,  $f\lesssim g$ means that $f\le
Cg$ for some positive constant $C$. It should be clear from the
context which parameters the constant $C$ is independent of. For
example, the constant in \eqref{delta-comp} is understood to be
independent of $z'$ and $\tau$.)

Let $z^0$ be a fixed boundary point and let $V$ be a neighborhood
of $z^0$ such that its closure is contained in a coordinate patch.
Let $m$ be any positive integer. It follows from Proposition 1.1
in \cite{FornaessSibony89} that for any $z'\in V\cap b\Omega$,
after a possible shrinking of $V$, there exists a neighborhood
$U_{z'}$ of $z'$ and local holomorphic coordinates
$(z_1, z_2)$ centered at $z'$ and depending smoothly on $z'$ such
that in these coordinates
$$
U_{z'}\cap\Omega=\{z\in U_{z'} \mid \rho(z)=\Re z_2+\psi(z_1, \Im
z_2)<0\},
$$
where  $\psi(z_1, \Im z_2)$ has the form of

\begin{equation}\label{defining-1}
\psi(z_1, \Im z_2)= P(z_1)+ (\Im z_2) Q(z_1) +O\big(|z_1|^{2m+1} +|\Im
z_2||z_1|^{m+1}+|\Im z_2|^2|z_1|\big)
\end{equation}
with
\[
P(z_1)=\sum_{l=2}^{2m} \sum_{\substack{j+k=l
\\j, k>0}}  {a}_{jk}(z')z_1^j\bar z_1^k
\quad \text{and}\quad Q(z_1)=\sum_{l=2}^{m} \sum_{\substack{j+k=l
\\j, k>0}}  {b}_{jk}(z')z_1^j\bar z_1^k
\]
being polynomials without harmonic terms.  Furthermore, there exist
positive constants $C_1$ and $C_2$, independent of $z'$, such that
\[
C_1A_l(z')\le \Big(\sum_{\substack{j+k\le l\\ j, k>0}} |
a_{jk}(z')|^2\Big)^{1/2} \le C_2 A_l(z')
\]
for $2\le l\le 2m$.

The above properties hold without the pseudoconvex or finite type
assumption on $\Omega$. Under the assumption that $b\Omega$ is pseudoconvex of
finite type $2m$, it then follows from Proposition 1.6 in
\cite{FornaessSibony89} that for all $0<\tau<1$,
\begin{equation}\label{FS-2}
\sum_{l=2}^m {B}_l(z')\tau^l\lesssim \tau (\delta(z',
\tau))^{1/2},
\end{equation}
where
\[
{B}_l(z')=\Big(\sum_{\substack{j+k\le l\\ j, k>0}} |
b_{jk}(z')|^2\Big)^{1/2}.
\]

The anisotropic bidisc $R_\tau(z')$ is given in the $(z_1,
z_2)$-coordinates by
\begin{equation}\label{eq:anisotropic}
R_\tau(z')=\{|z_1|<\tau, |z_2|<\delta(z', \tau)^{1/2}\}.
\end{equation}
We refer the reader to \cite{Catlin89, Mcneal89, NRSW89, Fu05} and
references therein for a discussion of these and other anisotropic
``balls". It was shown in \cite{Fu05} (see Lemmas 3.2 and 3.3
therein) that the anisotropic bidiscs $R_\tau(z')$ satisfy the
following doubling and engulfing properties: There exists a
positive constant $C$, independent of $z'$, such that if $z''\in
R_\tau(z')\cap b\Omega$, then $C^{-1}\delta(z', \tau)\le
\delta(z'', \tau)\le C\delta(z', \tau)$, $R_\tau(z')\subset
R_{C\tau}(z'')$, and $R_{\tau}(z'')\subset R_{C\tau}(z')$.

\subsection{Interior estimates}\label{sec:interior} Let $E$ be a
holomorphic line bundle over $\Omega$ that extends smoothly to the
boundary $b\Omega$. Let $e_k(\lambda; z, w)$ be the spectral
kernel of the $\dbar$-Neumann Laplacian on $(0, 1)$-forms on
$\Omega$ with values in $E^k$. Let $\pi\colon U\to b\Omega$ be the
projection onto the boundary such that $|r(z)|=\distance(z,
\pi(z))$. Shrinking $U$ if necessary, we have $\pi\in
C^\infty(U)$. Write $\tau_k=1/\sqrt{k}$.

\begin{proposition}\label{prop:kernel-estimate}
For any $C, c>0$,
\begin{equation}\label{eq:kernel-estimate}
\trace e_k(Ck; z, z)\lesssim  k(\delta(\pi(z), \tau_k))^{-1},
\end{equation}
for all sufficiently large $k$ and all $z\in \Omega$ with $d(z)\ge c
(\delta(\pi(z), \tau_k))^{1/2}$.
\end{proposition}

It is a consequence of a classical result of G\"{a}rding
\cite{Garding53} that for any compact subset $K$ of $\Omega$,
\begin{equation}\label{eq:garding}
\trace e_k(Ck; z, z)\lesssim k^2, \qquad \text{for}\ z\in K.
\end{equation}
Evidently, the constant in the above estimate depends on $K$. (See Theorem 3.2 in \cite{Berman04} for a more general and precise
version of this result.) In fact, this is also true for any
$z\in\Omega$ with $d(z)\ge 1/\sqrt{k}$. (Compare estimate (2.3) in
\cite{Metivier81}, Theorem 3.2 in \cite{Berman04}, and Proposition
5.7 in \cite{Berman05}.) Therefore, it suffices to establish
\eqref{eq:kernel-estimate} on $\{z\in\Omega, c(\delta(\pi(z),
\tau_k))^{1/2}\le |r(z)|\le \tau_k\}$. This also follows from the
elliptic theory, via an anisotropic rescaling. We provide details
below.

Let $z'\in b\Omega$. Following the discussion in
Section~\ref{sec:finite-type}, we can choose holomorphic
coordinates centered and orthonormal at $z'$ such that in a
neighborhood $U_{z'}$ of $z'$, $b\Omega$ is defined by $\rho(z_1,
z_2)=\Re z_2+\psi(z_1, \Im z_2)$ where $\psi(z_1, \Im z_2)$ is in the
form of \eqref{defining-1}.  Assume that the hermitian metric is
given on $U_{z'}$ by
\begin{equation}\label{eq:base-metric}
h=\frac{i}{2}\sum_{j, l=1}^2 h_{jl}(z) dz_j\wedge d\bar z_l, \quad
\text{with}\quad h_{jl}(0)=\delta_{jl}
\end{equation}
and the fiber metric on $E$ is given by
 \begin{equation}\label{eq:fiber-metric}
 \quad |e(z)|^2=e^{-\varphi(z)} \qquad \text{with}\quad \varphi
 (z)=\sum_{j, l=1}^2 a_{jl}z_j\bar z_l +O(|z|^3),
 \end{equation}
where $e(z)$ is an appropriate holomorphic frame of $E$ over
$U_{z'}$.

Write $\Omega_{z'}=\Omega\cap U_{z'}$. Let
$\omega'_2=\partial\rho$ and $\omega'_1=\rho_{\bar z_2}d
z_1-\rho_{\bar z_1}d z_2$. Let $\omega_2$ and $\omega_1$ be the
orthonormal basis for $(1, 0)$-forms on $\Omega_{z'}$ obtained by
applying the Gram-Schmidt process to $\omega'_2$ and $\omega'_1$.
Let $L_2$ and $L_1$ be the dual basis for $T^{1, 0}(\Omega_{z'})$.

Write $\delta_k=\delta(z', \tau_k)$ and let $c>0$. For any
$\sigma_k$ such that $\tau_k \ge \sigma_k\ge c\delta^{1/2}_k$, we
define the anisotropic dilation $(z_1, z_2)=F_k(\zeta_1,
\zeta_2)=(\tau_k\zeta_1, \ \sigma_k \zeta_2)$. Let
$U^k_{z'}=F^{-1}_k(U_{z'})$ and
$\Omega^k_{z'}=F^{-1}_k(\Omega_{z'})$. On $\Omega^k_{z'}$, we use
the base metric given by
\[
h^{(k)}=\frac{i}{2}\sum_{j, l=1}^2 h_{jl}(\tau_k\zeta_1,
\sigma_k\zeta_2) d\zeta_j\wedge d\bar \zeta_l
\]
and on $E^{(k)}=F_{k*}(E^k)$ we use the fiber metric given by the
weight function
\[
\varphi^{(k)}(\zeta)=k \varphi(\tau_k\zeta_1, \sigma_k\zeta_2).
\]
Note that $\Omega^k_{z'}=\{(\zeta_1, \zeta_2)\in U^k_{z'} \mid
\rho_k(\zeta_1, \zeta_2)<0\}$, where $\rho_k(\zeta_1,
\zeta_2)=(1/\sigma_k)\rho(\tau_k z_1, \sigma_k z_2)$. Let
$\omega^k_1$ and $\omega^k_2$ be the orthonormal basis for $(1,
0)$-forms on $\Omega^k_{z'}$ obtained as in the proceeding
paragraph but with $\rho$ replaced by $\rho_k$ and $(z_1, z_2)$
replaced by $(\zeta_1, \zeta_2)$ respectively. Let $L^k_1$ and
$L^k_2$ be the dual basis for $T^{1, 0}(\Omega^k_{z'})$. We define
$\scriptf_k\colon L^2(\Omega_{z'}, E^k)\to L^2(\Omega^k_{z'},
E^{(k)})$ by

\[ \scriptf_k(v)(\zeta_1,
\zeta_2)=(\tau_k\sigma_k)v(\tau_k\zeta_1, \sigma_k\zeta_2)
\]
and extend $\scriptf_k$ to act on forms by acting componentwise as
follows:
\[
\scriptf_k(v_1\ov{\omega}_1+v_2\ov{\omega}_2)=
\scriptf_k(v_1)\ov{\omega}^k_1+ \scriptf_k(v_2)\ov{\omega}^k_2,
\quad \scriptf_k(
v\ov{\omega}_1\wedge\ov{\omega_2})=\scriptf_k(v)\ov{\omega}^k_1
\wedge\ov{\omega}^k_2.
\]
(Hereafter, we identify a form with values in $E^k$ with its representation
in the given local holomorphic trivialization.) It is easy to see
that $\scriptf_k$ is isometric on $L^2$-spaces with respect to
specified metrics:
\[
\|u\|^2_{h, k\varphi}=\|\scriptf_k u\|^2_{h^{(k)}, \varphi^{(k)}},
\]
where $\|\cdot\|_{h, k\varphi}$ denotes the $L^2$-norm with respect
to the base metric $h$ and the fiber metric $k\varphi$ and likewise
$\|\cdot\|_{h^{(k)}, \varphi^{(k)}}$ the $L^2$-norm with respect to
the base metric $h^{(k)}$ and the fiber metric $\varphi^{(k)}$. Let
\[
Q^{(k)}(u, v)=\tau^2_k Q^k_\Omega(\scriptf^{-1}_k u, \scriptf^{-1}_k
v)
\]
with $\dom(Q^{(k)})=\{v \mid \scriptf_k^{-1} v\in \dom(Q^k_\Omega)
\text{ and } \support \scriptf_k^{-1} v\subset U_{z'}\} $, where
$Q^k_\Omega=Q^{E^k}_{\Omega, 1}$ is the sesquilinear form
associated with the $\dbar$-Neumann Laplacian
$\square^k_\Omega=\square^{E^k}_{\Omega, 1}$ on $(0, 1)$-forms on
$\Omega$ with values in $E^k$. Let $\square^{(k)}$ be the
self-adjoint operator associated with $Q^{(k)}$.

Let $P'=\{\zeta\in U^{k}_{z'} \mid |\zeta_1|<1/2,
|\zeta_2+1|<1/2\}$. It is easy to see that for sufficiently large
$k>0$, $P'$ is a relatively compact subset of $\Omega^{k}_{z'}$.

\begin{lemma}\label{lemma:garding1} Let $u\in \dom(Q^{(k)})
\cap C^\infty_c(P')$. Then
\begin{equation}\label{eq:garding1}
\|u\|^2_{1, \Omega^{k}_{z'}}\lesssim Q^{(k)}(u,
u)+\|u\|^2_{h^{(k)}, \varphi^{(k)}},
\end{equation}
where $\|\cdot\|_{1, \Omega^{k}_{z'}}$ is the $L^2$-Sobolev norm
of order $1$ on $\Omega^k_{z'}$.
\end{lemma}

\begin{proof} Write $u=u_1\overline{\omega}^{(k)}_1+
u_2\overline{\omega}^{(k)}_2$ and $v=\scriptf_k^{-1}
(u)=v_1\overline{\omega}_1+v_2\overline{\omega}_2$. Since $v$ is
supported on $R'_k=F_k(P')=\{z\in U_{z'} \mid |z_1|<(1/2)\tau_k,
|z_2+\sigma_k|<(1/2)\sigma_k\}\subset\subset \Omega$, it follows
from integration by parts that
\[
Q^k_\Omega(v, v)+ k\|v\|^2_{h, k\varphi}\gtrsim \sum_{j, l=1}^2
\|\overline{L}_j v_l\|^2_{h, k\varphi}.
\]
Notice that on $R'_k$,
\[
|h_{jl}(z)-\delta_{jl}|\lesssim \tau_k.
\]
>From Section~\ref{sec:finite-type}, we obtain by direct
calculations that on $R'_k$,
\[
\overline{L}_1=(1+O(\tau_k))\frac{\partial}{\partial \bar
z_1}+O(\tau_k)\frac{\partial}{\partial\bar z_2}, \quad
\overline{L}_2=O(\tau_k)\frac{\partial}{\partial \bar
z_1}+(1+O(\tau_k))\frac{\partial}{\partial\bar z_2}.
\]
Thus,
\begin{align*}
Q^{(k)}(u, u)+\|u\|^2_{h^{(k)}, \varphi^{(k)}}
&=\tau_k^2(Q^k_\Omega(v, v)+k\|v\|^2_{h, k\varphi})\gtrsim
\sum_{j, l=1}^2\|\tau_k \overline{L}_j v_l\|^2_{h,
k\varphi}\\
&\gtrsim \sum_{j=1}^2 \big(\|\frac{\partial u_j}{\partial\bar
z_1}\|^2_{h^{(k)}, \varphi^{(k)}}+ \frac{\tau_k^2}{\sigma_k^2}
\|\frac{\partial u_j}{\partial\bar z_2}\|^2_{h^{(k)},
\varphi^{(k)}}\big) \ge \sum_{j, l=1}^2 \|\frac{\partial
u_j}{\partial\bar z_l}\|^2_{h^{(k)}, \varphi^{(k)}}.
\end{align*}
Since $|\varphi^{(k)}|\lesssim 1$ on $P'$ and $u$ is compactly
supported in $P'$, a simple integration by parts argument then
yields the estimate \eqref{eq:garding1}.
\end{proof}

We now complete the proof of
Proposition~\ref{prop:kernel-estimate}.  From
Lemma~\ref{lemma:garding1}, we know that $\square^{(k)}$ is
uniformly (independent of $k$) strong elliptic on $P'$. Let
$P''=\{\zeta\in P' \mid |\zeta_1|<1/4, |\zeta_2+1|<1/4\}$. Thus by
G\"{a}rding's inequality,
\begin{equation}\label{eq:garding2}
\|u\|_{2M, P''}\lesssim \|(\square^{(k)})^M u\|_{P'}+\|u\|_{P'}
\end{equation}
for any $u\in C^\infty_{0, 1}(\Omega^{k}_{z'}, E^{(k)})$, where
$\square^{(k)}=\tau^2_k\scriptf_k \square^k_\Omega\scriptf^{-1}_k$
acts formally.

Let ${\bf E}_k(\lambda)$ be the spectral resolution of
$\square^k_\Omega$, the $\dbar$-Neumann Laplacian on $\Omega$ on
$(0, 1)$-forms with values in $E^k$. Let $v\in {\bf
E}_k(Ck)(L^2_{0, 1}(\Omega, E^k))$ be of unit norm. Then for any
positive integer $M$,

\begin{equation}\label{eq:pr1}
\|(\square^k_\Omega)^M v\|^2_{\Omega, E^k}\le (Ck)^{2M}.
\end{equation}

Let $u_k=\scriptf_k(v')$, where $v'$ is the restriction of $v$ to
$U_{z'}$. Then
\[
(\square^{(k)})^M u_k =\tau_k^{2M} \scriptf_k(\square^{k}_\Omega)^M
v'.
\]
By \eqref{eq:pr1}, we have
\[
\|(\square^{(k)})^M u_k\|^2_{h^{(k)}, \varphi^{(k)}} \lesssim 1.
\]
We obtain from \eqref{eq:garding2} and the Sobolev embedding
theorem that
\[
\tau_k \sigma_k |v(0, -\sigma_k)|=|u_k(0, -1)|\lesssim 1.
\]
Thus by \eqref{eq:e-s-compare}, we have
\[
\trace e_k(Ck; (0, -\sigma_k), (0, -\sigma_k))\lesssim
(\tau_k\sigma_k)^{-2}\le k\delta_k^{-1}.
\]
Since the constant in this estimate is uniform as $z'$ varies the
boundary $b\Omega$ and $\sigma_k$ varies between $c\delta_k^{1/2}$
and $\tau_k$, we thus conclude the proof of
Proposition~\ref{prop:kernel-estimate}.

\subsection{Boundary estimates}\label{sec:boundary}

\subsubsection{Main boundary estimate} We shall establish the
following boundary estimate for the spectral kernel.

\begin{proposition}\label{prop:spectral-2}  Let $C>0$. For any $z'\in b\Omega$ and
sufficiently large $k$,
\begin{equation}\label{spectral-3}
\int_{R_{\tau_k}(z')\cap\Omega} \trace e_k(Ck; z, z) dV(z)\lesssim
(\delta(z', \tau_k))^{-1/2} .
\end{equation}
\end{proposition}

Recall that $\tau_k=1/\sqrt{k}$ and $R_{\tau_k}(z')$ is the
anisotropic bidisc given by \eqref{eq:anisotropic}. Assume
Proposition~\ref{prop:spectral-2} for a moment, we now prove
the sufficiency in Theorem~\ref{main-theorem}. In fact we shall
prove the following:

\begin{proposition}\label{prop:n-est} Let $\Omega\subset\subset X$
be a smoothly bounded pseudoconvex domain in a complex surface. Let $E$
be a holomorphic line bundle over $\Omega$ that extends smoothly to $b\Omega$. If $b\Omega$ is of finite type $2m$, then for any $C>0$, there exists $C'>0$ such that
 $N_k(Ck)\le C'k^{1+m}$. More precisely, $\lim_{k\to\infty} N_k(Ck)/k^{1+m}=0$ when $m>1$.
\end{proposition}

\begin{proof} We cover $b\Omega$ by finitely many open sets, each of which is contained in a coordinate patch as the $V$'s in
Section~\ref{sec:finite-type}. Let $z'\in V\cap b\Omega$.
Multiplying both sides of  \eqref{spectral-3} by $(\delta(z',
\tau_k))^{-1/2}$ and integrating with respect to $z'\in V\cap
b\Omega$, we obtain by the Fubini-Tonelli theorem that
\[
\int_{\Omega} \trace e_k(Ck;  z, z)\, dV(z) \int_{V\cap b\Omega}
\chi_{\Omega\cap R_{\tau_k}(z')}(z) (\delta(z', \tau_k))^{-1/2}\,
dS(z')\lesssim \int_{V\cap b\Omega} (\delta(z', \tau_k))^{-1}\,
dS(z').
\]
(Here $\chi_S$ denotes the characteristic function of the set
$S$.) By Lemma 3.4 in [Fu05], we then have
\begin{equation}\label{eq:e1}
\int_{V\cap \{z\in\Omega\mid d(z)<c(\delta(\pi(z),
\tau_k))^{1/2}\} } \trace e_k(Ck;  z, z)\, dV(z)\lesssim
\tau_k^{-2}\int_{V\cap b\Omega} (\delta(z', \tau_k))^{-1}\,
dS(z'),
\end{equation}
for some positive constant $c$.

On $\{z\in \Omega \mid d(z)\ge c(\delta(\pi(z), \tau_k))^{1/2}\}$,
we know from Proposition~\ref{prop:kernel-estimate} that
\begin{equation}\label{eq:e2}
\trace e_k (Ck; z, z)\lesssim k(\delta(\pi(z),
\tau_k))^{-1}\lesssim k^{1+m}.
\end{equation}
Also, on any relatively compact subset of $\Omega$, we have
\begin{equation}\label{eq:e3}
\trace e_k(Ck; z, z)\lesssim k^2,
\end{equation}
where the constant depends on the compact set (see
\eqref{eq:garding}). By definition, we have
\begin{equation}\label{eq:e4}
N_k(Ck)=\int_\Omega \trace e_k(Ck; z, z)\, dV(z).
\end{equation}
It then follows from \eqref{eq:e1}-\eqref{eq:e3} that
\[
N_k(Ck)\lesssim k^{1+m}.
\]
Note that
\begin{equation}\label{eq:e5}
\lim_{k\to\infty} \frac{\tau_k^{2m}}{\delta(z', \tau_k)}=0
\end{equation}
when $z'\in b\Omega$ is of type less than $2m$ and the set of
weakly pseudoconvex boundary points has zero surface measure.
Combining \eqref{eq:e1}-\eqref{eq:e5}, we obtain from the Lebesgue
dominated convergence theorem that $\limsup_{k\to\infty}
N_k(Ck) /k^{m+1}=0$ when $m>1$. \end{proof}

\noindent{\bf Remark}. Heuristic arguments seem to suggest that
the optimal estimates are $N_k(Ck)\lesssim k^m\ln k$ when $m=2$
and $N_k(Ck)\lesssim k^m$ when $m>2$.

The remaining subsections are devoted to prove
Proposition~\ref{prop:spectral-2}.  The proof follows along the
line of arguments of the proof of Lemma 6.2 in \cite{Fu05}: we
need to show here that the contributions from the curvatures of
the metric on the base $X$ and fiber metrics on $E^k$ are negligible.
We provide necessary details below.

\subsubsection{Uniform Kohn Estimate}\label{sec:kohn} We will use
a slight differently rescaling scheme from the one in the previous
section. Following \cite{Fu05}, we flatten the boundary before
rescaling the domain and the $\dbar$-Neumann Laplacian. Let $z'\in
b\Omega$. As in Section~\ref{sec:interior}, we may choose local
holomorphic coordinates $(z_1, z_2)$, centered and orthonormal at
$z'$, such that a defining function $\rho$ of $b\Omega$ in a
neighborhood $U_{z'}$ of $z'$ has the form given by
\eqref{defining-1}. Furthermore, we may assume that the base
metric $h$ on $X$ and the fiber metric $\varphi$ on $E$ are of the forms
\eqref{eq:base-metric} and \eqref{eq:fiber-metric} respectively.
Write
\[
\rho=\Re z_2+f(z_1)+(\Im z_2)g_1(z_1)+(1/2)(\Im z_2)^2
g_2(z_1)+O(|\Im z_2|^3),
\]
where
\[ f(z_1)= P(z_1)+O(|z_1|^{2m+1}),\ \ g_1(z_1)=
Q(z_1)+O(|z_1|^{m+1}),\ \ g_2(z_1)=O(|z_1|).
\]
Let
\[(\eta_1,
\eta_2)=\Phi_{z'}(z_1, z_2)=(z_1,\ z_2+ h(z_1, \Im z_2)-F(z_1,
z_2)),
\]
where
\[
F(z_1, z_2)=\frac{1}{2}g_2(z_1)(\Re z_2+  h(z_1, \Im
z_2))^2+i(g_1(z_1)(\Re z_2)+g_2(z_1)(\Re z_2) (\Im z_2)).
\]
(See \cite{Fu05}, Section~4.)

Let $\hat\rho(z)= \rho(z)- (1/2)g_2(z_1)( \rho(z))^2$. Then
$\hat\rho(z)$ is a also defining function for
$\Omega_{z'}=\Omega\cap U_{z'}$ near the origin.  Let $\omega_1$
and $\omega_2$ be an othonormal basis for $(1, 0)$-forms on
$U_{z'}$ obtained as in Section~\ref{sec:interior} but with $\rho$
replaced by $\hat\rho$. Let $L_1$ and $L_2$ be the dual basis for
$T^{1,0}(U_{z'})$.

 We now proceed with the rescaling.  Write $\delta=\delta(z', \tau)$.
 For any $\tau>0$, we define
\[
(w_1, w_2)=D_{z', \tau}(\eta_1, \eta_2)=(\eta_1/\tau,
\eta_2/\delta).
\]
Let  $\Phi_{z', \tau}=D_{z', \tau}\circ\Phi_{z'}$ and let
$\Omega_{z', \tau}=\Phi_{z', \tau}(\Omega_{z'})\subset\{(w_1,
w_2)\in \C^2 \mid \Re w_2<0\}$. (In what follows, we sometimes
suppress the subscript $z'$ for economy of notations when this
causes no confusions.) Let
\[
P_\tau(z')=\{|w_1|<1, \ |w_2|<\delta^{-1/2}\}.
\]
It is easy to see that $R_{C^{-1}\tau}(z')\subset \Phi^{-1}_{z',
\tau}(P_\tau(z'))\subset R_{C\tau}(z')$ (see Lemma 4.1 in
\cite{Fu05}). Let $\scriptg_{\tau}\colon (L^2(\Omega_{z',
\tau}))^2\to L^2_{(0, 1)}({\Omega}_{z'}, E^k)$ be the
transformation defined by
\[
\scriptg_{\tau} (u_1, u_2)=|\det d\Phi_\tau|^{1/2}\big(u_1
(\Phi_\tau)\ov{\omega}_1+ u_2(\Phi_\tau)\ov{ \omega}_2\big),
\]
where on $L^2(\Omega_{z', \tau})$ we use the standard Euclidean
metric and we identify as before forms with values in the line
bundle $E^k$ with its representation under the given holomorphic
trivialization. Let $\tau_k=1/\sqrt{k}$ and $\delta_k=\delta(z',
\tau_k)$ as before. Let
\begin{equation}\label{Q-tau-def}
Q_{\tau_k}(u, v)=\tau^2_k Q^k_\Omega(\scriptg_{\tau_k} u,
\scriptg_{\tau_k} v)
\end{equation}
be the densely defined, closed sesquilinear form on
$(L^2(\Omega_{\tau_k}))^2$ with
$\dom(Q_{\tau_k})=\{\scriptg^{-1}_{\tau_k}(u); \
u\in\dom(Q^k_\Omega), \ \supp u\subset U_{z'}\}$. Here, as before,
$Q^k_\Omega(\cdot, \cdot)$ is the sesquilinear form associated
with the $\dbar$-Neumann Laplacian $\square^k_\Omega$ on $L^2_{(0,
1)}(\Omega, E^k)$.

The following lemma play a crucial role in the analysis. It is a
consequence of Kohn's commutator method and is the analogue of
Lemma~4.5 in \cite{Fu05}. We use $\vvv \cdot \vvv^2_\eps$ to
denote the tangential Sobolev norm of order $\eps>0$ on
$\C^2_{-}=\{(w_1, w_2)\in \C^2\mid \Re w_2<0\}$.

\begin{lemma}\label{lm:Q-estimate}  There exists an $\eps>0$ such that
for any sufficiently large $k$,
\begin{equation}\label{eq:Q-estimate}
Q_{\tau_k}(u, u)+ \|u\|^2\gtrsim \vvv u
\vvv^2_\eps+\tau_k^2\delta_k^{-2} \vvv \frac{\partial u}{\partial
\bar w_2}\vvv^2_{-1+\eps},
\end{equation}
for all $u\in\dom(Q_{\tau_k})\cap C^\infty_c(P_{\tau_k}(z'))$.
\end{lemma}

\begin{proof} Let $v=\scriptg_k u=v_1\overline{\omega}_1+
v_2\overline{\omega}_2$.  It is easy to see that
\[
|k\varphi (z)|\lesssim 1, \qquad |h_{jl}(z)-\delta_{jl}|\lesssim
\tau_k
\]
on $R_{\tau_k}(z')$. Thus
\[
\|u\|^2\approx \|v\|^2_{h, k\varphi}\approx
\|v_1\|^2_0+\|v_2\|^2_0,
\]
where $\|\cdot\|_0$ denotes the $L^2$-norm corresponding to the
standard Euclidean metric on $U_{z'}$ in the $(z_1,
z_2)$-variables. We will also use $dV_0$ and $dS_0$ to denote the
volume and surface elements in the Euclidean metric.

By integration by parts, we have (see \cite{Hormander65, Kohn72}),
\begin{equation}\label{hormander}
Q^k_\Omega(v, v)+k\|v\|^2_{h, k\varphi}\gtrsim k\|v\|^2_{h,
k\varphi}+ \sum_{j,l=1}^2\|\ov{L}_j v_l\|^2_{h,
k\varphi}+\int_{b\Omega} \big(\partial\dbar\hat\rho(L_1,
\ov{L}_1)\big)|v|^2 e^{-k\varphi}\, dS.
\end{equation}
Therefore,
\begin{align*}
Q_{\tau_k}(u, u) +\|u\|^2 &\gtrsim \tau^2_k (Q^k_\Omega (v, v)+
k\|v\|^2_{h, k\varphi})\\
&\gtrsim \tau^2_k\big(k\|v\|^2_0+ \sum_{j, l=1}^2\|\ov{L}_j
v_l\|^2_0+\int_{b\Omega} \big(\partial\dbar\hat\rho(L_1,
\ov{L}_1)\big)|v|^2\, dS_0 \big) \\
&\gtrsim \tau^2_k\big(k\|v\|^2_0+ \sum_{j, l=1}^2\|\ov{L}_j
v_l\|^2_0+\sum_{l=1}^2\|L_1 v_j\|^2_0\big) \\
&\gtrsim \|u\|^2+ \sum_{j, l=1}^2\|\ov{L}_{j, \tau_k} u_l\|^2+
\sum_{l=1}^2\|{L}_{1, \tau_k} u_l\|^2,
\end{align*}
where $\ov{L}_{j, \tau_k}=\tau_k(\Phi_{z', \tau_k})_*(\ov{L}_j)$.
The estimate \eqref{eq:Q-estimate} then follows from (the proof of
) Lemma~4.5 in \cite{Fu05}. \end{proof}

Once Lemma~\ref{lm:Q-estimate} is established, the proof of
Proposition~\ref{prop:spectral-2} follows along the lines of the
proof of Lemma~6.2 in \cite{Fu05}. Since there are necessary
modifications due to the possible present of the $\dbar$-cohomology,
we provide the necessary details for completeness in the next subsection.

\subsubsection{Comparison with an auxiliary Laplacian}\label{sec:aux}
Let $\eps$ be the order of the Sobolev norm in
Lemma~\ref{lm:Q-estimate}. Let $W_{\eps, \delta_k}$ be the space
of all $u\in L^2(\C^2_{-})$ such that
\begin{equation}\label{m-tau-norm}
\|u\|^2_{\eps, \delta_k}= \vvv u \vvv^2_{\eps}+\delta^{-1}_k\vvv
\frac{\partial u}{\partial\bar w_2}\vvv^2_{-1+\eps} <\infty.
\end{equation}
Let ${\square}_{\eps, \delta_k}$ be the associated densely
defined, self-adjoint operator on $L^2(\C^2_-)$ such that
$\|u\|^2_{\eps, \delta_k}=\|\square_{\eps, \delta_k}^{1/2} u\|^2$
and $\dom(\square^{1/2}_{\eps, \delta_k})=W_{\eps, \delta_k}$. Let
${N}_{\eps, \delta_k}$ be its inverse. Let $\chi (w_1, w_2)$ be a
smooth cut-off function supported on $\{|w_1|<2, |w_2|<2\}$ and
identically 1 on $\{|w_1|<1, |w_2|<1\}$.  Let
$\chi_{\delta_k}(w_1, w_2)=\chi(w_1, \delta^{1/2}_k w_2)$. We use
$\lambda_j(T)$ to denote the $j^{\text{th}}$ singular value
(arranged in a decreasing order and repeated according to
multiplicity) of a compact operator $T$. It then follows from the min-max
principle that
\begin{equation}\label{eq:min-max-t}
\lambda_{j+k+1}(T_1+T_2)\le \lambda_{j+1}(T_1)+\lambda_{k+1}(T_2)
\quad \text{and}\quad \lambda_{j+k+1}(T_3\circ T_1)\le \lambda_{j+1}(T_1)
\lambda_{k+1}(T_3)
\end{equation}
(see \cite{Weidmann80}).

For sufficiently large
$k$ and $j$, we then have
\begin{equation}\label{eq:N-estimate}
\lambda_j(\chi_{\delta_k}{N}^{1/2}_{\eps, \delta_k})\lesssim
(1+j\delta^{1/2}_k)^{-\eps/4}
\end{equation}
(see Lemma~5.3 in \cite{Fu05}).

Let $\kappa$ be a cut-off function compactly supported on $U_{z'}$
and identically 1 on a neighborhood of $z'$ of uniform size.  Let
\[
{\bf E}_{\tau_k}(\lambda)=\scriptg^{-1}_{\tau_k}\kappa {\bf
E}_k(\lambda/\tau_k^2)\kappa\scriptg_{\tau_k}\colon
(L^2({\Omega}_{\tau_k}))^2\to (L^2({\Omega}_{\tau_k}))^2.
\]
The kernel of ${\bf E}_{\tau_k}(\lambda)$ is then given by
\begin{equation}\label{eq:spectral-1}
e_{\tau_k}(\lambda; w, w')=e(\lambda/\tau^2_k; \Phi^{-1}_{\tau_k}
(w), \Phi^{-1}_{\tau_k}(w'))\kappa(w)\kappa(w')|\det
d\Phi^{-1}_{\tau_k}(w)|^{\frac{1}{2}}|\det
d\Phi^{-1}_{\tau_k}(w')|^{\frac{1}{2}}.
\end{equation}

We now proceed to prove Proposition~\ref{prop:spectral-2}. By
\eqref{eq:spectral-1}, it suffices to prove that
\begin{equation}\label{spectral-4}
\int_{P_{\tau_k}(z')\cap {\Omega}_{\tau_k}} \trace e_{\tau_k}(C;
w, w)\, dV_0(w) \lesssim \delta_k^{-1/2}.
\end{equation}
Let $\square_{\tau_k}\colon (L^2({\Omega}_{\tau_k}\cap
P_{\tau_k}(z')))^2\to (L^2({\Omega}_{\tau_k}\cap
P_{\tau_k}(z')))^2$ be the operator associated with the
sesquilinear form $Q_{\tau_k}$ given by \eqref{Q-tau-def} but with
domain
\[\dom(Q_{\tau_k})=\{\scriptg^{-1}_{\tau_k}(u) \mid u\in
\dom(Q^k_\Omega), \supp u\subset
\Phi^{-1}_{\tau_k}(P_{\tau_k}(z'))\}.
\]
Thus $\square_{\tau_k}=\tau^2_k \scriptg_{\tau_k}^{-1}\square
\scriptg_{\tau_k}$. Let $N_{\tau_k}=(I+\square_{\tau_k})^{-1}$. It
follows from Lemma~\ref{lm:Q-estimate} that
\[
Q_{\tau_k}(u, u)+\|u\|^2\gtrsim \|u\|^2_{\eps, \delta_k},
\]
for any $u\in \dom(\square^{1/2}_{\tau_k})$. Therefore,
\[
\|u\|^2=Q_\tau(N^{1/2}_{\tau_k} u, N^{1/2}_{\tau_k}
u)+\|N^{1/2}_{\tau_k} u\|^2\gtrsim \|{\square}^{1/2}_{\eps,
\delta_k} N^{1/2}_{\tau_k} u\|^2.
\]
Thus $N^{1/2}_{\tau_k}=\chi_{\delta_k}N^{1/2}_{\tau_k}
=\chi_{\delta_k}{N}^{1/2}_{\eps, \delta_k}{\square}^{1/2}_{\eps,
\delta_k}N^{1/2}_{\tau_k}$. It follows from \eqref{eq:min-max-t}
and \eqref{eq:N-estimate} that
\begin{equation}\label{n-tau}
\lambda_j(N_{\tau_k}^{1/2}) \lesssim
\lambda_j(\chi_{\delta_k}{N}_{\eps, \delta_k}^{1/2})\lesssim
(1+j\delta^{1/2}_k)^{-\eps/4}.
\end{equation}

Let $K$ be any positive integer such that $K>4/\eps$.  Let
$\chi^{(j)}$, $j=0, 1, \ldots, K$, be a family of cut-off
functions supported in $\{|w_1|<1,  |w_2|<1\}$ such that
$\chi^{(0)}=\chi$ and $\chi^{(j+1)}=1$ on $\supp \chi^{(j)}$. Let
\[
{\bf E}^{(l)}_{\tau_k}(\lambda)=\scriptg^{-1}_{\tau_k} \kappa
(\tau^2_k\square^k_\Omega)^l {\bf
E}_k(\lambda\tau_k^{-2})\kappa\scriptg_{\tau_k}\colon
(L^2({\Omega}_{\tau_k}))^2\to (L^2({\Omega}_{\tau_k}))^2.
\]
Thus ${\bf E}^{(0)}_{\tau_k}(\lambda)={\bf E}_{\tau_k}(\lambda)$
and
\begin{equation}\label{e-ell}
\|{\bf E}^{(l)}_{\tau_k} (C) u\|\lesssim \|u\|.
\end{equation}
Furthermore,
\begin{equation}\label{q-q}
Q_{{\tau_k}}(\chi^{(j)}_{\delta_k} {\bf E}^{(l)}_{\tau_k}(C)u)
={\tau_k}^2 Q^k_\Omega(\chi^{(j)}_{\delta_k}(\Phi_{{\tau_k}})
({\tau_k}^2\square^k_\Omega)^l {\bf
E}(C{\tau_k}^{-2})\kappa\scriptg_{{\tau_k}} u)
\end{equation}
(Here we use $Q(u)$ to denote $Q(u, u)$ for abbreviation.)

It is straightforward to check that
\begin{equation}\label{Q-commute}
Q^k_\Omega(\theta u, \theta u)=\Re \langle\langle\theta
\square^k_\Omega u, \theta u\rangle\rangle+(1/2)\langle\langle u,
[\theta, A] u\rangle\rangle
\end{equation}
where $\theta$ is any smooth function on $\overline{\Omega}$ and
$A=[\dbar^*, \theta]\dbar+\dbar [\dbar^*, \theta]+\dbar^*[\dbar,
\theta]+[\dbar, \theta]\dbar^*$. (Here $\dbar^*=\dbar^*_{h,
k\varphi}$ is the adjoint of $\dbar$ with respect to the base
metric $h$ and fiber metric $k\varphi$.) Note that $[\theta, A]$
is of zero order and its sup-norm is bounded by a constant
independent of $k$.

It follows from \eqref{q-q}, \eqref{Q-commute}, and the Schwarz
inequality that for any $u\in (L^2({\Omega}_{\tau_k}))^2$,
\begin{equation}\label{q-q-2}
Q_{{\tau_k}}(\chi^{(l')}_{\delta_k} {\bf E}^{(l)}_{\tau_k}(C)u)
\lesssim \|\chi^{(l')}_{\delta_k} {\bf E}^{(l+1)}_{\tau_k} (C)
u\|^2+\|\chi^{(l'+1)}_{\delta_k} {\bf E}^{(l)}_{\tau_k} (C) u\|^2.
\end{equation}
Hence
\begin{align*}\label{q-q-3}
\|(I+\square_{\tau_k})^{1/2}\chi^{(l')}_{\delta_k} {\bf
E}^{(l)}_{\tau_k}(C)u\|^2&=Q_{{\tau_k}}(\chi^{(l')}_{\delta_k}
{\bf E}^{(l)}_{\tau_k}(C)u)+\|\chi^{(l')}_{\delta_k} {\bf
E}^{(l)}_{\tau_k}(C)u\|^2\\
&\lesssim \|\chi^{(l')}_{\delta_k} {\bf E}^{(l+1)}_{\tau_k} (C)
u\|^2+\|\chi^{(l'+1)}_{\delta_k} {\bf E}^{(l)}_{\tau_k} (C) u\|^2
+\|\chi^{(l')}_{\delta_k} {\bf
E}^{(l)}_{\tau_k}(C)u\|^2\\
\end{align*}
It then follows from \eqref{eq:min-max-t} that
\begin{equation}\label{lambda-j}
\begin{aligned}
\lambda_{4j+1}(\chi^{(l')}_{\delta_k} {\bf E}^{(l)}_{\tau_k}
(C))&\le \lambda_{j+1}
(N_{\tau_k}^{1/2})\lambda_{3j+1}((I+\square_{\tau_k})^{1/2}_{\tau_k}
\chi^{(l')}_{\delta_k} {\bf E}^{(l)}_{\tau_k} (C))\\
&\le \lambda_{j+1}
(N_{\tau_k}^{1/2})\big(\lambda_{j+1}(\chi^{(l')}_{\delta_k} {\bf
E}^{(l+1)}_{\tau_k} (C)) +\lambda_{j+1}(\chi^{(l'+1)}_{\delta_k}
{\bf E}^{(l)}_{\tau_k}(C)) \\
&\qquad +\lambda_{j+1}(\chi^{(l')}_{\delta_k} {\bf
E}^{(l)}_{\tau_k}(C)) \big).
\end{aligned}
\end{equation}
Using \eqref{n-tau}, \eqref{e-ell}, and\eqref{lambda-j}, we then
obtain by an inductive argument on $K-(l+l')$ that
\begin{equation}\label{lambda-e-ell}
\lambda_j(\chi^{(l')}_{\delta_k} {\bf E}^{(l)}_{\tau_k}
(C))\lesssim (1+j{\delta_k}^{1/2})^{-(K-(l+l'))\eps/4}
\end{equation}
for any pair of non-negative integers $l, l'$ such that $0\le
l+l'\le K$ and for all $j>C_1\delta_k^{-1/2}$, where $C_1$ is a sufficiently
large constant. In particular,
\[
\lambda_j(\chi_{\delta_k} {\bf E}_{\tau_k}(C))\lesssim
(1+j{\delta_k}^{1/2})^{-K\eps/4}.
\]
Since ${\bf E}_{\tau_k}(C)$ has uniformly bounded operator norms, we also have that
$\lambda_j(\chi_{\delta_k} {\bf E}_{\tau_k}(C))\lesssim 1$.  The
trace norm of $\chi_{\delta_k} {\bf E}_{\tau_k} (C)$ is then given
by
\[
\sum_{j\lesssim{\delta_k}^{-\frac{1}{2}}} \lambda_j(\chi_{\delta_k}
{\bf E}_{\tau_k}(C))+\sum_{j\gtrsim{\delta_k}^{-\frac{1}{2}}}
\lambda_j(\chi_{\delta_k} {\bf E}_{\tau_k}(C)) \lesssim
{\delta_k}^{-\frac{1}{2}}+ \sum_{j\gtrsim{\delta_k}^{-\frac{1}{2}}}
(1+j{\delta_k}^{-\frac{1}{2}})^{-\frac{K\eps}{4}} \lesssim
{\delta_k}^{-\frac{1}{2}}.
\]
Inequality \eqref{spectral-4} is now an easy consequence of the
above estimate.

\subsection{Estimate of the type}\label{sec:est-type}
In this section, we prove the necessity in Theorem~\ref{main-theorem}. More precisely, we prove the following:

\begin{proposition}\label{prop:est-type}
Let $\Omega\subset\subset X$ be a smoothly bounded pseudoconvex domain
in a complex surface. Let $E$ be holomorphic line bundle over
$\Omega$ that extends smoothly to $b\Omega$. Let $M>0$. If for any $C>0$, there exists $C'>0$ such that $N_k(Ck)\le C' k^M$ for all sufficiently large integer $k$, then the type of the domain of $b\Omega$ is $\le 8M$.
\end{proposition}

The proof of the above proposition, using a wavelet construction of Lemari\'{e} and Meyer \cite{LemarieMeyer86}, is a modification of the proof of Theorem~1.3 in \cite{Fu05}.
We provide the full details below. We begin with the following simple well-known consequence of the min-max principle.

\begin{lemma}\label{lm:min-max} Let $Q$ be a semi-positive, closed, and densely defined
sesquilinear form on a Hilbert space. Let $\square$ be the associated densely defined self-adjoint operator operator. Let $\{u_l \mid 1\le l\le k\}\subset\dom(Q)$.  Let $\lambda_l$ and $\tilde\lambda_l$ be the $l^{\text{th}}$-eigenvalues of operator $\square$ and the hermitian matrix $(Q(u_j, u_l))_{1\le j, l\le k}$ respectively. If
\[
\|\sum_{l=1}^k c_l u_l\|^2\ge C\sum_{l=1}^k |c_l|^2
\]
for all $c_l\in \C$, then $\tilde\lambda_l\ge C\lambda_l$, $1\le l\le k$.
\end{lemma}

\begin{proof} By the min-max principle,
\[
\tilde\lambda_l=\inf\{\tilde\lambda (\tilde L) \mid  \tilde L \text{ is an $l$-dimensional subspace of $\C^k$}\}
\]
where
\[
\tilde\lambda(\tilde L)=\sup\{\sum_{j, l=1}^k c_j \bar c_l Q(u_j,
u_l) \mid  (c_1, \ldots, c_k)\in \tilde L, \sum_{l=1}^k |c_l|^2=1\}.
\]
Likewise,
\[
\lambda_l=\inf_{\substack{L\subset \dom(Q)\\ \dim(L)=l}}\sup\{Q(u, u) \mid
u\in L, \|u\|=1\}.
\]
For any $l$-dimensional subspace $\tilde L$ of $\C^k$,
let $L=\{\sum_{l=1}^k c_l u_l \mid  (c_1, \ldots, c_k)\in \tilde
L\}$ and let
\[
\lambda(L)=\sup\{Q(u, u) \mid  u\in L, \|u\|=1\}.
\]
Then $\tilde\lambda (\tilde L)\ge C_2\lambda(L)$.  Hence
$\tilde\lambda_l\ge C_2\lambda_l$ for all $1\le l\le k$.
\end{proof}

Let $z'\in b\Omega$. We follow the notations and setup as in the proof of Proposition~\ref{prop:kernel-estimate}.  Suppose the type of $b\Omega$ is
$\ge 2m$ at $z'$. Then $P(z_1)=O(|z_1|^{2m})$. It follows from \eqref{FS-2} that
$Q(z_1)=0$. Hence
\begin{equation}\label{eq:h-est}
\psi(z_1, \Im z_2)=O(|z_1|^{2m}+|\Im z_2||z_1|^{m+1}+|\Im z_2|^2|z_1|).
\end{equation}

Let $b(t)$ be a smooth function
supported in $[-1/2, \ 1]$ such that $b(t)= 1$ on $[0, \ 1/2]$ and $b^2(t)+b^2(t-1)=
1$ on $[1/2, \ 1]$. It follows that  $\{b(t) e^{2\pi l t i} \mid l\in\Z\}$ is an orthogonal system\footnote{In fact, it was shown by Lemari\'{e} and Meyer \cite{LemarieMeyer86} that the Fourier transform of $b(t)$ is a wavelet} in $L^2(\R)$ (\cite{LemarieMeyer86}; see also \cite{Daubechies88, HernandezWeiss96}).
Write $z_2=s+it$. Let $\chi$ be any smooth cut-off function supported on $(-2,\ 2)$ and identically 1 on $(-1, 1)$ and let
\[
B(z_2)=(b(t)-ib'(t)s-b''(t)s^2/2)\chi(s/(1+|t|^2)).
\]
Then $B(0,t)=b(t)$ and $|\partial B(z_2)/\partial\bar z_2|\lesssim
|s|^2$. Let $a(z_1)$ be a smooth function identically 1 on $|z_1|\le 1/2$ and
supported on the unit disc. For any
positive integers $j$ and for any positive integer $l$ such that
$2^{mj-1}/j\le l\le 2^{mj}/j$, let
\[
u_{j, l}(z)=l^{1/2}2^{2(m+1)j} a(2^{2j}z_1)B(2^{2mj}z_2)e^{2\pi l 2^{2mj}z_2}e^{k\varphi(z)/2}(g(z))^{-1/2}\overline{\omega}_1,
\]
where $g(z)=\det(h_{jl}(z))$. For any sufficiently large $j$, $u_{j, l}$ is a compactly
supported smooth $(0, 1)$-form in $\dom(Q^k_{\Omega})$. (Recall that $Q^k_\Omega$ is
the sesquilinear form associated with the $\dbar$-Neumann Laplacian $\square^k_\Omega$ on $\Omega$ for $(0, 1)$-forms with values in $E^k$.)
Moreover, after the substitutions $(z_1, z_2)\to (2^{-2j} z_1, 2^{-2mj}z_2)$, we have
\[
\|u_{j, l}\|^2_{h, k\varphi}\
=l\int_{\C} |a(z_1)|^2 dV_0(z_1)\int_\R \,dt\int^{-2^{2mj}\psi(2^{-2j}z_1, 2^{-2mj}t)}_{-\infty}
|B(z_2)|^2 e^{4\pi l s}\, ds.
\]
By \eqref{eq:h-est}, $|2^{2mj}\psi(2^{-2j}z_1, 2^{-2mj}t)|\lesssim 2^{-2mj}$.
Note also that $l 2^{-2mj}\le j^{-1}2^{-mj}$. Thus
\begin{align*}
\|u_{j, l}\|^2_{h, k\varphi} &\gtrsim l\int_{\C} |a(z_1)|^2
dV_0(z_1)\int_\R \,d t\int^{-C 2^{-2mj}}_{-\infty} |B(z_2)|^2 e^{4\pi l s}\, d s\\
&\gtrsim \int_{\C}|a(z_1)|^2 dV_0(z_1)\int_0^{1/2} dt
\int_{-1}^{-C2^{-2mj}} l e^{4\pi l s}\, ds\gtrsim 1.
\end{align*}
Similarly, $\|u_{j, l}\|^2_{h, k\varphi}\lesssim 1$, and hence $\|u_{j,
l}\|^2_{h, k\varphi} \approx 1$. Also, for any $l$ and $l'$ such that $2^{mj-1}/j\le l, l'
\le 2^{mj}/j$,
\[
\langle\langle u_{j, l}, u_{j, l'}\rangle\rangle_{h, k\varphi} =\sqrt{ll'} \int_{\C} |a(z_1)|^2 dV_0\int_\R dt
\int^{-2^{2mj}\psi(2^{-2j}z_1, 2^{-2mj}t)}_{-\infty}
|B(z_2)|^2 e^{2\pi((l+l')s+i(l-l')t)}\, ds
\]
We decompose the above integral into two parts.
Let $A$ be the above expression  with the upper limit in the last
integral over $s$ replaced by $0$ but keep the lower limit.
Let $B$ likewise be the expression with the lower limit replaced by
$0$ but keep the upper limit. Hence $\langle\langle u_{j, l}, u_{j, l'}\rangle\rangle_{h, k\varphi}=A+B$.
We first estimate $B$:
\begin{align*}
|B| &\le \sqrt{ll'} \int_{\C} |a(z_1)|^2 dV_0\int_\R
d t \int^{-2^{2mj}\psi(2^{-2j}z_1, 2^{-2mj}t)}_0
|B(z_2)|^2
e^{2\pi(l+l')s}\, d s \\
&\lesssim \frac{\sqrt{ll'}}{l+l'}\left(1-e^{C(l+l')2^{-2mj}}\right)
\lesssim j^{-1}2^{-mj}.
\end{align*}
To estimate $|A|$, we use the orthogonality of the system of functions
$\{b(t)e^{2\pi lt i} \mid l\in
\Z\}$ in $L^2(\R)$.  It follows that if $l\not= l'$, then
\[
A=\sqrt{ll'}\int_{\C} |a(z_1)|^2 dV_0 \int_\R d
t\int^0_{-\infty} \left(|B(s, t)|^2-|B(0, t)|^2\right)e^{2\pi((l+l')s+i(l-l')t)} \,ds.
\]
Therefore,
\[
|A|\lesssim \sqrt{ll'}\int_{\C} |a(z_1)|^2 dV_0
\int_{-1}^1\,d t\int^0_{-\infty}  s
e^{2\pi(l+l')s}\, d s \lesssim \sqrt{ll'}/(l+l')^2\lesssim j^{-1}2^{-mj}.
\]
For sufficiently large $j$ and for any $k$, $l$ such that
$2^{mj-1}/j\le l, l' \le 2^{mj}/j$, $l\not= l'$, we then have,
\[
|\langle\langle u_{j, l},\ u_{j, l'}\rangle\rangle_{h, k\varphi} |\lesssim j^{-1}2^{-mj}.
\]
For any $c_l\in\C$, we have
\begin{align*}
\|\sum_l c_l u_{j, l}\|^2_{h, k\varphi} &=\sum_{l} |c_l|^2\|u_{j,
l}\|^2_{h, k\varphi}-\sum_{\substack{l, l'\\ l\not=l'}}
c_l\overline{c}_{l'} \langle\langle u_{j, l}, \ u_{j, l'}\rangle\rangle_{h, k\varphi} \\
&\ge \sum_{l} |c_l|^2\|u_{j, l}\|^2_{h, k\varphi}-j^{-1}2^{-mj}\big|\sum_{l}
c_l\big|^2 \\
&\gtrsim (1-j^{-2})\sum_{l} |c_l|^2 \gtrsim \sum_{l} |c_l|^2,
\end{align*}
where the summations are taken over all integers $l$ between
$2^{mj-1}/j$ and $2^{mj}/j$.

Write $L^{k\varphi}_1=e^{k\varphi}L_1 e^{-k\varphi}$. Then
\begin{equation}
Q^k_{\Omega}(u_{j, l}, u_{j, l}) \lesssim \|u_{j,
l}\|^2_{h, k\varphi}+\|\overline{L}_2 u_{j, l}\|^2_{h, k\varphi}+\|L^{k\varphi}_1 u_{j, l}\|^2_{h, k\varphi}.
\end{equation}
Recall that
\[
\varphi(z)=O(|z|^2) \quad \text{and}\quad |h_{jl}(z)-\delta_{jl}|\lesssim |z|.
\]
Moreover,
\begin{align*}
\overline{L}_2&=O(|z|)\frac{\partial}{\partial\bar z_1}+O(1)\frac{\partial}{\partial \bar z_2}\\
\intertext{and}
L_1&=O(1)\frac{\partial}{\partial z_1}+O(|z_1|^{2m-1}+|\Im z_2||z_1|^m+|\Im z_2|^2)\frac{\partial}{\partial z_2}.
\end{align*}
Let $k=2^{4j}$.  It follows that when $|z_1|\lesssim 2^{-2j}$ and $|z_2|\lesssim 2^{-2mj}$,
\[
k|\varphi(z)|\lesssim 1 \quad \text{and}\quad k|\nabla\varphi(z)|\lesssim 2^{2j}.
\]
Therefore, on the one hand,
\begin{align*}
\|\overline{L}_2 u_{j, l}\|^2_{h, k\varphi}&\lesssim \||z|\frac{\partial u_{j, l}}{\partial\bar z_1}\|^2_{h, k\varphi}+\|\frac{\partial u_{j, l}}{\partial\bar z_2}\|^2_{h, k\varphi}\\
&\lesssim  2^{4j}+ 2^{4mj}\int_{\C}|a(z_1)|^2\,dV_0\int_{-1}^1\,d t \int^{-2^{2mj}
\psi(2^{-2j}z_1, 2^{-2mj}t)}_{-\infty} l\big|\frac{\partial
B}{\partial \bar z_2}\big|^2 e^{4\pi l s}\,d s\\
&\lesssim 2^{4j}+2^{4mj}\int_{\C}|a(z_1)|^2\,dV_0\int_{-1}^1\,dt \int^{C2^{-2mj}}_{-\infty} l s^4 e^{4\pi l s}\,d s\\
&\lesssim 2^{4j}+2^{4mj}l^{-4}\lesssim 2^{4j}.
\end{align*}
On the other hand,
\begin{align*}
\|L^{k\varphi}_1 u_{j, l}\|^2_{h, k\varphi}&\lesssim \|(k\nabla \varphi)u_{j, l}\|^2_{h, k\varphi}+\big\|\frac{\partial u_{j,
l}}{\partial z_1}\big\|^2_{h, k\varphi}+\big\|(|z_1|^{2m-1}+|t||z_1|^m+t^2)\frac{\partial
u_{j, l}}{\partial z_2}\big\|^2_{h, k\varphi}\\
&\lesssim 2^{4j}+2^{-4mj+4j} \int_{\C} |a(z_1)|^2
dV_0\int_{-1}^1\, d t \int^{C2^{-2mj}}_{-\infty}   l^3e^{4\pi l s}\, ds\\
&\lesssim 2^{4j}.
\end{align*}
We thus have
\begin{equation}\label{eq:q-k}
Q^k_{\Omega}(u_{j, l}, u_{j, l})\lesssim 2^{4j}.
\end{equation}
Let $\lambda_{k, l}$ be the $l^{\text{th}}$ eigenvalues of $\square^k_\Omega$. The hypothesis of Proposition~\ref{prop:est-type} implies that
\begin{equation}\label{eq:q-k-2}
\lambda_{k, l}\ge Ck
\end{equation}
when $l>C'k^M$. Proving by contradiction, we suppose that $M<m/4$. It follows from Lemma~\ref{lm:min-max} and \eqref{eq:q-k-2} that,
\begin{equation}\label{eq:sum-est}
\sum_{l=j^{-1}2^{mj-1}}^{j^{-1}2^{mj}}Q^k_\Omega(u_{j, l}, u_{j, l}) \ge
\sum_{l=1}^{j^{-1}2^{mj-1}} \lambda_{k, l}\ge C(j^{-1}2^{mj-1}-C'k^M)k.
\end{equation}
Combining \eqref{eq:q-k} and \eqref{eq:sum-est}, we then have
\[
j^{-1}2^{mj-1} 2^{4j} \gtrsim C(j^{-1}2^{mj-1}-C'k^M)k.
\]
Dividing both sides by $k=2^{4j}$ and $j^{-1}2^{mj-1}$, we obtain
\[
1\gtrsim C(1-C' 2^{(4M-m)j+1}).
\]
Since by assumption, $C$ can be chosen arbitrarily large, we
arrive at a contradiction by letting $j\to\infty$. We thus conclude the proof of Proposition~\ref{prop:est-type}.

\bibliography{survey}
%

\providecommand{\bysame}{\leavevmode\hbox
to3em{\hrulefill}\thinspace}

\end{document}